\documentclass[12pt,oneside]{ip-journal}

\usepackage[bookmarks=true, bookmarksopen=true,%
bookmarksdepth=3,bookmarksopenlevel=2,%
colorlinks=true,%
linkcolor=blue,%
citecolor=blue,%
filecolor=blue,%
menucolor=blue,%
urlcolor=blue]{hyperref}

\usepackage{authblk}
\usepackage{stmaryrd}
\usepackage{amsmath}
\usepackage{amsfonts}
\usepackage{amssymb}
\usepackage{amsthm}
\usepackage{mathabx}
\usepackage[all,2cell]{xy}
\usepackage{color}
\usepackage{etoolbox}
\usepackage{tikz-cd}
\usepackage{relsize}
\usepackage{enumitem}
\usepackage[margin=1.5in]{geometry}
\usepackage{xifthen}
\usepackage{amsbsy}

\UseTwocells

\newbool{hidedetails}
\booltrue{hidedetails}

\newbool{hideextra}
\booltrue{hideextra}

\newbool{hidefla}
\booltrue{hidefla}

\newtheorem{proposition}{Proposition}
\newtheorem{corollary}{Corollary}

\newtheorem*{theo}{Theorem}

\theoremstyle{definition}

\newtheorem{remark}{Remark}

\newcommand{\red}{\color{red}}
\newcommand{\blue}{\color{blue}}
\newcommand{\green}{\color{green}}

\newcommand{\hide}[1]{\ifbool{hidedetails}{}{{\blue #1}}}
\newcommand{\extra}[1]{\ifbool{hideextra}{}{{\green #1}}}
\newcommand{\fla}[1]{\ifbool{hideextra}{#1}{{\red #1}}}

\newcommand{\lp}{\left(}
\newcommand{\rp}{\right)}
\newcommand{\lf}{\left\{}
\newcommand{\rf}{\right\}}

\newcommand{\hooklongrightarrow}{\lhook\joinrel\longrightarrow}			

\newcommand{\ainty}{{\mathcal A_\infty}}				
\newcommand{\sset}[1][]{{\rm{SSet}_{#1}}}		
\newcommand{\coss}[1][]{{\rm{c\sset[#1]}}}		
\newcommand{\stasi}[1][]{{\boldsymbol\Delta_{#1}}}		
\newcommand{\stasic}[1][]{{\stasi[#1]\lp\mathbb C\rp}}		
\newcommand{\dgcon}[1][]{{\mathfrak C^{\leq 0}\lp{#1}\rp}}		
\newcommand{\dgcop}[1][]{{\ifthenelse{\isempty{#1}}{\mathfrak C^{\geq 0}}{\mathfrak C^{\geq 0}\lp{#1}\rp}}} 
\newcommand{\ddcom}[3][]{{\ifthenelse{\isempty{#1}}{\mathfrak C^{#2,#3}}{\mathfrak C^{#2,#3}\lp{#1}\rp}}} 
\newcommand{\dgcou}[2][]{{\ifthenelse{\isempty{#1}}{\mathfrak C^{#2}}{\mathfrak C^{#2}\lp{#1}\rp}}} 
\newcommand{\mixing}{{\mu}}		
\newcommand{\grmix}[3][]{{\ifthenelse{\isempty{#1}}{\mathfrak M^{#2,#3}}{\mathfrak M^{#2,#3}\lp{#1}\rp}}}	
\newcommand{\grix}[1][]{{XXX}}
\newcommand{\codeg}[1][]{{s}}		
\newcommand{\snorm}{{N}}	
\newcommand{\cnorm}{{\mathcal N}}	
\newcommand{\cmorm}{{\widetilde{\mathcal N}}}	
\newcommand{\sig}{{\iota}} 

\newcommand{\GL}[1]{{{\rm GL}_{#1}\lp\mathbb C \rp}}																		
\newcommand{\smgl}[1]{{\underline{\rm GL}_{#1}\lp \mathbb C \rp}}		
\newcommand{\gl}[1]{{\mathfrak{gl}\lp #1\rp}}				
\newcommand{\G}[1][]{{{\mathcal G}_{#1}}}							
\newcommand{\PGL}[1]{{{\rm PGL}_{#1}\lp \mathbb C\rp}}						
\newcommand{\smpgl}[1]{{\underline{\rm PGL}_{#1}\lp \mathbb C \rp}}		
\newcommand{\SL}[1]{{{\rm SL}_{#1}\lp \mathbb C\rp}}			
\newcommand{\algr}{{\rm G}}																				    
\newcommand{\subg}{{\rm H}}		
\newcommand{\quospa}{{\mathcal M}}		
\newcommand{\quobit}{{\mathcal U}}		
\newcommand{\allie}{{\mathfrak g}}			
\newcommand{\dete}[1][]{{\ifthenelse{\isempty{#1}}{\Delta}{\Delta(#1)}}}						  
\newcommand{\spec}[1][]{{\ifthenelse{\isempty{#1}}{\rm Spec}{{\rm Spec}({#1})}}}				 
\newcommand{\proj}[1][]{{\ifthenelse{\isempty{#1}}{\rm Proj}{{\rm Proj}(#1)}}}					
\newcommand{\gring}[1][\bullet]{{A_{#1}}}																	      
\newcommand{\quot}{{\rm Quot}}																					    
\newcommand{\quotof}[2]{{\quot\lp #1,#2 \rp}}																		     
\newcommand{\derqu}{{\rm DQuot}}																				  
\newcommand{\derqof}[4][\bullet]{{\derqu^{#1}_{#4}\lp #2,#3 \rp}}		 
\newcommand{\derquz}[1][]{{{\rm Gr}_{#1}}}												
\newcommand{\grass}[2]{{{\rm Grass}\lp #1,#2 \rp}}				  
\newcommand{\twish}[2][]{{\ifthenelse{\isempty{#1}}{\mathcal O_{#2}}{\mathcal O_{#2}(#1)}}}	
\newcommand{\twi}[2]{{ #1 \lp #2 \rp}}												                                       
\newcommand{\posi}{{\mathbb N}}											                                           
\newcommand{\coh}{{\mathcal F}}											                                            
\newcommand{\sumto}[2]{{#1^{\oplus^{#2}}}}							                                         

\newcommand{\dimglo}{{n}}						                                                                      
\newcommand{\dimgra}[1][]{{N_{#1}}}							
\newcommand{\dimpro}[1]{{l_{#1}}}							
\newcommand{\CY}{{\mathcal X}}						                                                              
\newcommand{\hilpo}[1][]{{{\rm P}_{#1}}}						                                                 
\newcommand{\casmu}{{m}}														                                    
\newcommand{\addshi}{{k}}																							  
\newcommand{\lebo}{{p}}																							 
\newcommand{\ribo}{{q}}																							  
\newcommand{\tribun}[1]{{\mathcal W_{#1}}}																	            
\newcommand{\tabu}[1][]{{\mathcal V_{#1}}}																		 
\newcommand{\Homo}[2]{{{\mathcal Hom_0}\lp #1,#2\rp}}		
\newcommand{\biglie}[1][\bullet]{{\mathfrak g^{#1}}}			
\newcommand{\bigring}[2][\bullet]{{\mathcal C^{#1}_{#2}}}			
\newcommand{\matri}{{M}}			

\newcommand{\staquz}[1][]{{\mathbf{Gr}_{#1}}}		
\newcommand{\semiquz}[1][]{{{Gr}_{#1}}}		
\newcommand{\staqof}[4][\bullet]{{\mathbf{DQuot}^{#1}_{#4}\lp #2,#3 \rp}}		
\newcommand{\derquot}{{\mathbf{DQuot}^{\bullet}}}	
\newcommand{\demquot}[1][\bullet]{{\underline{\mathbf{DQuot}}^{#1}}}	
\newcommand{\pt}{{\rm p}}		
\newcommand{\secti}{{\sigma}}			
\newcommand{\quo}[2]{{{#1}\sslash{#2}}}			

\newcommand{\dera}[2][\bullet]{{{#2}^{#1}}}						
\newcommand{\cinfty}{{C^\infty}}							
\newcommand{\dedi}{{\mathfrak d}}							
\newcommand{\homdi}{{\delta}}			
\newcommand{\nform}[2][]{{\ifthenelse{\isempty{#1}}{\omega_{#2}}{\omega_{#2,#1}}}}						
\newcommand{\symfo}{{\underline{\omega}}}					
\newcommand{\symstra}[1][]{{\boldsymbol\omega^{#1}}}					
\newcommand{\formstra}[2][]{{\ifthenelse{\isempty{#1}}{\boldsymbol\omega_{#2}}{\boldsymbol\omega_{#2,#1}}}}
\newcommand{\refo}[1]{{{#1}_{\rm re}}}						
\newcommand{\imfo}[1]{{{#1}_{\rm im}}}						
\newcommand{\para}{{\tau}}									
\newcommand{\sara}{{\sigma}}			
\newcommand{\negacy}[2][2]{{{{\rm NC}^{#2}_{#1}}}}				
\newcommand{\sish}[3][]{{{\widetilde{\boldsymbol{\mathfrak F}}}^{#3}_{#1}\lp #2 \rp}}	
\newcommand{\lish}[3][]{{{\overline{\boldsymbol{\mathfrak F}}}^{#3}_{#1}\lp #2 \rp}}	
\newcommand{\nish}[3][]{{{{\boldsymbol{\mathfrak N}}}^{#3}_{#1}\lp #2 \rp}}	

\newcommand{\dish}[3][]{{{\boldsymbol{\mathfrak F}}^{#3}_{#1}\lp #2 \rp}}				
\newcommand{\derham}[2][]{{\Omega^{#1}_{#2}}}			
\newcommand{\staclo}[2][\bullet]{{\boldsymbol\Omega^{#1}_{2,#2}}}		
\newcommand{\striclo}[2][\bullet]{{\overline{\boldsymbol\Omega}^{#1}_{2,#2}}}		
\newcommand{\suspense}[2][1]{{#2[#1]}}						
\newcommand{\shift}[2][1]{{#2\langle #1\rangle}}						
\newcommand{\Sym}[3][]{{{\rm Sym}_{#1}^{#2}\lp #3 \rp}}				
\newcommand{\hoclo}[2][2]{{\mathcal A_{#1}^{#2}}}		
\newcommand{\astack}{{\mathfrak S}}		
\newcommand{\quack}[2]{{\llbracket #1/#2\rrbracket}}		
\newcommand{\duack}[2]{{\Big\langle #1/#2\Big\rangle}}		
\newcommand{\tanga}[2][\bullet]{{\mathbb T_{#2}^{#1}}}					

\newcommand{\lagra}[1][]{{\boldsymbol{\lambda}}}					
\newcommand{\lagrast}[2][0]{{\boldsymbol\lambda}_{#2,#1}}					
\newcommand{\lmin}[1][\bullet]{{\mathbb L^{#1}}}							
\newcommand{\llin}[1][\bullet]{{\mathbb K^{#1}}}			
\newcommand{\dishe}[2][]{{{\boldsymbol{\mathfrak L}}_{#1}\lp #2 \rp}}		
\newcommand{\fol}[2][]{{{\boldsymbol{\mathfrak F}}_{#1}\lp #2 \rp}}		
\newcommand{\anchor}{{\alpha}}							
\newcommand{\coanchor}{{\boldsymbol\alpha}}		
\newcommand{\gramial}[1][\bullet]{{\Lambda^{#1}}}		
\newcommand{\alldis}[1]{{{\mathfrak L}\lp #1\rp}}			
\newcommand{\alldig}[1]{{{\mathfrak L}^{\rm gr}\lp #1\rp}}		
\newcommand{\stadis}{{\boldsymbol{\mathfrak L}}}		
\newcommand{\derma}[1][\bullet]{{\underline{U}^{#1}}}		
\newcommand{\derna}[1][\bullet]{{\underline{V}^{#1}}}		
\newcommand{\deroa}[1][\bullet]{{\underline{W}^{#1}}}		
\newcommand{\smoup}{{\smgl{\dimglo}}}		
\newcommand{\spoup}{{\smpgl{\dimglo}}}	
\newcommand{\dersch}[1][]{{{\rm S}^{#1}}}			
\newcommand{\dgsch}[1][\bullet]{{{\mathbf M}^{#1}}}			
\newcommand{\noneg}{{\mathbb Z_{\geq 0}}}		
\newcommand{\map}[3][]{{{\rm Map}_{#1}\lp #2,#3\rp}}		
\newcommand{\hocolim}[1][]{{{\rm hocolim}\,}}		
\newcommand{\holim}[1][]{{\underset{#1}{\rm holim}\,}}		
\newcommand{\stadi}{{\Delta}}		
\newcommand{\saff}[1][]{{\rm DAff_{#1}}}						
\newcommand{\shi}{{d}}		

\newcommand{\kerre}[1][]{{{\mathcal K}^{#1}}}		
\newcommand{\relide}[1][]{{{\mathfrak I}_{#1}}}			
\newcommand{\irre}[1][]{{{\mathfrak T}_{#1}}}		

\newcommand{\nerve}[1]{{{\mathbf N}\lp #1\rp}}		
\newcommand{\Ker}[1]{{{\rm Ker}\lp #1\rp}}		

\newcommand{\vect}[1][\mathbb C]{{{\rm Vect}_{#1}}}		
\newcommand{\hyperco}[1][\bullet]{{{\mathbb H}^{#1}}}		
\newcommand{\kan}[1][-]{{{\mathfrak K}\lp #1\rp}}	
\newcommand{\negdi}{{\boldsymbol\Lambda}}	
\newcommand{\eva}{{\rm ev}}		

\begin{document}
	
\title[Global shifted potentials for moduli stacks]{Global shifted potentials for moduli stacks of sheaves on Calabi-Yau four-folds}

\maketitle
\smallskip
\author{Dennis Borisov, Ludmil Katzarkov, Artan Sheshmani, Shing-Tung Yau}
{Dennis Borisov${}^{1}$, Ludmil Katzarkov${}^{4, 6,7}$, Artan Sheshmani${}^{2,3,4}$ and  Shing-Tung Yau$^{2,5}$}

\address{${}^1$  Department of Mathematics and Statistics, University of Windsor, 401 Sunset Ave, Windsor Ontario, Canada}

\address{${}^2$ Harvard University CMSA and Physics department, Jefferson Laboratory, 17 Oxford St, Cambridge, MA 02138}
\address{${}^3$ Institut for Matematik , Aarhus Universitet, Ny Munkegade 118 Building 1530, DK-8000 Aarhus C, Denmark}
\address{${}^4$ National Research University Higher School of Economics, Russian Federation, Laboratory of Mirror Symmetry, NRU HSE, 6 Usacheva str.,Moscow, Russia, 119048}
\address{${}^5$ Department of Mathematics, Harvard University, Cambridge, MA 02138, USA}
\address{${}^6$ University of Miami, Coral Gables, FL}
\address{${}^7$ Institute of Mathematics and Informatics, Bulgarian Academy of Sciences}

\date{\today}

\begin{abstract}
It is shown that there are globally defined Lagrangian distributions on the stable loci of derived $\quot$-stacks of coherent sheaves on Calabi--Yau four-folds. Dividing by these distributions produces perfectly obstructed smooth stacks with globally defined $-1$-shifted potentials, whose derived critical loci give back the stable loci of smooth stacks of sheaves in global Darboux form.

\medskip
	
\noindent{\bf MSC codes:} 14A20, 14N35, 14J35, 14F05, 55N22, 53D30

\noindent{\bf Keywords:} Calabi--Yau four-folds, moduli stack of stable sheaves, Derived quotient stacks, Shifted symplectic structures, Invariant Lagrangian distributions, Global shifted potentials. 

\end{abstract}

\tableofcontents

\section*{Introduction}

This paper is concerned with construction of Lagrangian distributions for $-2$-shifted symplectic structures. Just as in \cite{BSY} the shifted symplectic structures are $\mathbb C$-linear, while Lagrangian distributions are $\mathbb R$-linear. The Lagrangian condition is satisfied with respect to the imaginary part of the $\mathbb C$-linear symplectic structure, and, in addition, there is a negative definiteness requirement with respect to the real part. 

These choices originate in the notion of anti-self-dual instantons on Calabi--Yau $4$-folds proposed by Donaldson and Thomas in \cite{DT}. The idea of using shifted symplectic structures defined by Pantev, To\"en, Vaqui\'e and Vezzosi in \cite{PTVV13}, is due to Dominic Joyce, and it was partially implemented in \cite{BoJ13}.

The main result of this paper is the following.

\begin{theo} For any Calabi--Yau $4$-fold, and any choice of a Hilbert polynomial, the corresponding moduli stack of stable points in the derived $Quot$-scheme with the action of a general linear group can be equipped with a globally defined invariant distribution, that is Lagrangian with respect to the imaginary part and negative definite with respect to the real part of the $-2$-shifted symplectic form.\end{theo}

\smallskip

\subsection*{Background} In \cite{BSY} the authors have shown that given a derived scheme equipped with a $-2$-shifted symplectic structure, one can construct a Lagrangian distribution defined on the entire scheme. Locally this problem is trivial due to the local Darboux theorem, proved by Brav, Bussi and Joyce in \cite{BBJ13}. So the main result of \cite{BSY} is that these local constructions can be glued globally.

In this paper we consider derived Artin stacks equipped with $-2$-shifted symplectic structures. Examples of such stacks are given by quotient stacks of derived $\quot$-schemes by the actions of $\GL{\dimglo}$. In \cite{BKS} it is shown that these quotient stacks do carry shifted symplectic structures, obtained in a canonical way from the big stack of all perfect complexes on the Calabi--Yau manifold in question.

\smallskip

Working with derived $\quot$-schemes and their quotients allows us to use some of the results from Geometric Invariant Theory. In particular we use the fact that in the \'etale topology stable loci of $\quot$-schemes are principal $\PGL{\dimglo}$-bundles over the good quotients. There is a simple extension of this result to derived $\quot$-schemes, and, together with the stacky version of the local Darboux theorem from \cite{BBBJ}, we manage to reduce the problem of constructing Lagrangian distributions on the stable loci of derived quotient stacks to the problem of constructing such distributions on derived schemes. This is the problem the authors have solved in \cite{BSY}.

There is another advantage in working with quotient stacks of derived $\quot$-schemes. These schemes, defined by Ciocan-Fonatine and Kapranov in \cite{DerQuot}, are what we call dg manifolds. They are given by enhancing the structure sheaf of a smooth quasi-projective scheme to a sheaf of differential non-positively graded algebras. Having one explicitly defined quasi-projective scheme on which everything happens proves to be very useful for our purposes.

\smallskip

\subsection*{Applications} There are various reasons one might be interested in constructing a Lagrangian distribution on a derived scheme equipped with $-2$-shifted symplectic structure. One is to take the quotient and obtain a perfectly obstructed derived $\cinfty$-manifold. This manifold might be oriented as in \cite{CGJ}, and then, assuming the manifold is compact, one can hope to produce some invariants.

In this paper, as in \cite{BSY}, we have a different goal. The quotient by a Lagrangian distribution is not just a perfectly obstructed $\cinfty$-manifold, or perfectly obstructed stack in our case. As it is explained in \cite{P14} this quotient carries a $-1$-shifted potential, i.e.\@ a section of the bundle dual to the bundle of obstructions. Derived critical locus of this potential reconstructs the entire moduli stack we have started with (in its $\cinfty$-version).

It is here that having just one smooth quasi-projective scheme underlying the derived $\quot$-scheme proves so useful. Once we have established that there is a globally defined Lagrangian distribution on the quotient stack, we can write this distribution as a globally defined $\GL{\dimglo}$-invariant subcomplex of the tangent complex of the derived $\quot$-scheme. After dividing by this subcomplex we are left with the same smooth quasi-projective variety with the same action of $\GL{n}$, and with a $\GL{n}$-linearized bundle together with a $\GL{n}$-invariant section and a $\GL{n}$-invariant co-section. This simple set of data encodes the entirety of the quotient stack.\footnote{A somewhat similar direction of research, but using very different methods, is pursued by Thomas and Oh, \cite{Oh}.}Moreover, constructing the corresponding derived critical locus of the co-section gives us this stack \emph{globally} in a Darboux form.

\smallskip

Existence of this globally defined potential might turn out to be important in its own right. As was pointed out in \cite{P14} it is often interesting to have moduli spaces realized as critical loci. Usually these are critical loci of functions, albeit defined possibly on infinite dimensional manifolds. In our case we realize a moduli stack as a critical locus not of a function, but of a $-1$-shifted function -- the co-section. 

Being shifted this function is not defined on a manifold, but on a perfectly obstructed manifold, which according to Uhlenbeck--Yau theorem \cite{UY86} corresponds to the moduli space of ${\rm Spin}(7)$-instantons. So in fact we have a layering of moduli problems, where moduli of ${\rm SU}(4)$-connections appear as critical loci on the moduli of ${\rm Spin}(7)$-connections. Here we follow \cite{Conan} in classifying connections by normed division algebras.

\smallskip

\subsection*{Contents of the paper} In Section \ref{one} we recall the construction of derived $\quot$-schemes and provide the details on actions of $\GL{\dimglo}$, noting that this action factors through $\PGL{\dimglo}$ also in the derived case. We finish the section showing that with respect to the \'etale topology the stable locus is a principal $\PGL{\dimglo}$-bundle.

In Section \ref{twoone} we analyze homotopically closed forms on quotient stacks. Spaces of such forms can be described as cosimplicial objects in the category of simplicial vector spaces, where cosimplicial dimension corresponds to the simplicial dimension in the bar construction of the action. We show that using cosimplicial-simplicial normalization we can describe homotopically closed forms on quotient stacks as cocycles in dg vector spaces consisting of formal power series, where the power of the formal parameter corresponds to the cosimplicial dimension. This is similar to the formal power series arising from de Rham stacks in the theory of homotopically closed forms on derived schemes, as defined in \cite{PTVV13}.

Section \ref{twotwo} begins with an analysis of the strictly invariant homotopically closed forms on quotient stacks. If we would like to work with homotopically closed forms on a quotient of a derived scheme $\dgsch$ with respect to an action of a group $\algr$, we might want to look for such forms on $\dgsch$ itself, and see what conditions they need to satisfy and what additional structures to have, to describe forms on the quotient stack. A (strict) invariance would be obviously needed and some pairing with the Lie algebra of the group would also be required. We give a precise definition and show that, in case $\algr$ is linearly reductive, any homotopically closed form on the quotient stack can be written in the strictly invariant way.

The main part of Section \ref{twotwo} is dedicated to analysis of integrable distributions on quotient stacks and isotropic structures on them, if the stacks are equipped with homotopically closed forms. Our starting point is the definition in terms of stacks of $\infty$-categories by To\"en and Vezzosi found in \cite{AlgFoliations}, which we strictify, i.e.\@ make strictly invariant with respect to the group action, in order to be able to deal with these objects effectively. The result is a simple description of isotropic structures in terms of formal power series.

Finally in Section \ref{three} we use the machinery from the previous section to reduce the problem of constructing Lagrangian distributions on the quotient stacks of derived $\quot$-schemes coming from Calabi--Yau four-folds to the problem solved in \cite{BSY}. This section ends with some very simple linear algebra showing existence of globally defined subcomplexes of the tangent complex on the derived $\quot$-scheme that represent these distributions.

\subsection*{Acknowledgments}  \includegraphics[scale=0.2]{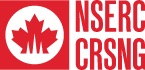} The first author acknowledges the support of the Natural Sciences and Engineering Research Council of Canada (NSERC), [RGPIN-2020-04845].

Cette recherche a été financée par le Conseil de recherches en sciences naturelles et en génie du Canada (CRSNG), [RGPIN-2020-04845].

\smallskip

The first author would like to thank Dominic Joyce, Tony Pantev and Dingxin Zhang for very helpful conversations. The second author was partially supported by NSF Grant, Simons Investigatior  Award HMS, Simons Collaboration Award HMS, National Science Fund of Bulgaria, National Scientific Program “Excellent Research and People for the Development of European
Science” (VIHREN), Project No. KP-06-DV-7, HSE University Basic Research Program.The third author would like to thank Dennis Gaitsgory, Amin Gholampour, Martijn Kool, Naichung Conan Leung, and Tony Pantev for many helpful discussions and commenting on the first versions of this article. Research of A.S. and D. B. was partially supported by generous Aarhus startup research grant of A. S., and partially by the NSF DMS-1607871, NSF DMS-1306313, the Simons 38558, and Laboratory of Mirror Symmetry NRU HSE, RF Government grant, ag. No 14.641.31.0001. A.S. would like to further sincerely thank the Center for Mathematical Sciences and Applications at Harvard University, the center for Quantum Geometry of Moduli Spaces at Aarhus University, and the Laboratory of Mirror Symmetry in Higher School of Economics, Russian federation, for the great help and support. S.-T. Y. was partially supported by NSF DMS-0804454,
NSF PHY-1306313, and Simons 38558.

\section{Derived $\quot$-schemes and actions of $\GL{\dimglo}$}\label{one}

Let $\CY$ be a reduced\extra{\footnote{This requirement ensures that spaces of global sections of twisted sheaves have the expected dimensions, e.g.\@ $\dim(\Gamma(\CY,\twish\CY))$ is the number of connected components.}} connected\extra{\footnote{On a disconnected scheme the problem of moduli of sheaves breaks into separate problems, one for each connected component. Hence it is natural to assume existence of only one connected component.}} projective scheme over $\mathbb C$, equipped with a very ample line bundle $\twish[1]\CY$,\extra{\footnote{According to the standard definition (e.g.\@ Hartshorne)  for a morphism of schemes to be projective is a \emph{property}, not a structure. A particular choice of a very ample line bundle is a structure.}} and let $\hilpo\in\mathbb Q[t]$. As the family of \fla{semi-stable} coherent sheaves on $\CY$ with Hilbert polynomial $\hilpo$ is bounded (e.g.\@ \cite{HuyLehn} Thm.\@ 3.3.7 p.\@ 78), there is $\casmu\in\posi$, s.t.\@ any such sheaf $\coh$ is $\casmu$-regular (\cite{Kleiman} Thm.\@ 1.13, p.\@ 623). From $\casmu$-regularity of $\coh$ it follows that $\forall i\geq 0$ $\twi{\coh}{i+\casmu}$ is globally generated and 
	\begin{equation}\label{surjectivity}\Gamma\lp\CY,\twish[i]\CY\rp\otimes\Gamma\lp\CY,\twi{\coh}{\casmu}\rp\longrightarrow\Gamma\lp\CY,\twi{\coh}{i+\casmu}\rp\end{equation} 
is surjective (\cite{Mumford} p.\@ 100). In particular we can find any such $\coh$ among quotients of $\twi{\sumto{\twish{\CY}}{\dimglo}}{-\casmu}$, where $\dimglo:=\hilpo(\casmu)$. So we need to consider the $\quot$-scheme $\quotof{\twi{\sumto{\twish{\CY}}{\dimglo}}{-\casmu}}{\hilpo}$ (\cite{GrothIV} Thm.\@ 3.2, p.\@ 260). The kernels of the quotients $\twi{\sumto{\twish{\CY}}{\dimglo}}{-\casmu}\twoheadrightarrow\coh$ do not have to be globally generated, however, the family of all these kernels is bounded (\cite{GrothIV} Prop.\@ 1.2, p.\@ 252), and hence $\casmu'$-regular for some $\casmu'\in\posi$. Therefore there is $\lebo\in\posi$ s.t.\@ $\forall\addshi\geq\lebo$, $\twi{\sumto{\twish{\CY}}{\dimglo}}{-\casmu}$ is $\addshi$-regular as well as any $\coh$ as above, and the kernel of $\twi{\sumto{\twish{\CY}}{\dimglo}}{-\casmu}\twoheadrightarrow\coh$. Then surjectivity of maps as in (\ref{surjectivity}) gives us for each $\addshi\geq\lebo$ a realization of  $\quotof{\twi{\sumto{\twish{\CY}}{\dimglo}}{-\casmu}}{\hilpo}$ as a closed subscheme
	\begin{equation}\quotof{\twi{\sumto{\twish{\CY}}{\dimglo}}{-\casmu}}{\hilpo}\hooklongrightarrow\grass{\dimgra[\addshi]-\hilpo(\addshi)}{\dimgra[\addshi]},\end{equation} 
where $\dimgra[\addshi]:=\dimglo\hilpo[{\twish{\CY}}](\addshi-\casmu)$ and $\grass{\dimgra[\addshi]-\hilpo(\addshi)}{\dimgra[\addshi]}$ is the Grassmannian of $\dimgra[\addshi]-\hilpo(\addshi)$-dimensional subspaces in an $\dimgra[\addshi]$-dimensional space (\cite{GrothIV} Lemmas 3.3 p.\@ 261 and 3.7 p.\@ 264).

\bigskip

The action of $\GL{\dimglo}$ on $\sumto{\twish{\CY}}{\dimglo}$ induces an action of $\GL{\dimglo}$ on $$\Gamma\lp \CY,\twi{\sumto{\twish{\CY}}{\dimglo}}{\addshi-\casmu}\rp$$ where each $\dimglo\times\dimglo$ matrix becomes a matrix of $\hilpo[\twish{\CY}](\addshi-\casmu)\times\hilpo[\twish{\CY}](\addshi-\casmu)$ scalar matrices.\footnote{In particular, scalar $\dimglo\times\dimglo$-matrices are mapped to scalar matrices.} Thus we have a right\footnote{We regard global sections of $\sumto{\twish{\CY}}{\dimglo}$ as row vectors, i.e.\@ the action of $\GL{\dimglo}$ is from the right.} action of $\GL{\dimglo}$ on $\grass{\dimgra[\addshi]-\hilpo(\addshi)}{\dimgra[\addshi]}$. The Pl\"ucker embedding 
	\begin{equation}\grass{\dimgra[\addshi]-\hilpo(\addshi)}{\dimgra[\addshi]}\hooklongrightarrow\mathbb P^{\dimpro{\addshi}},\quad \dimpro{\addshi}={\dimgra[\addshi]\choose{\hilpo(\addshi)}}-1\end{equation}
comes with a $\GL{\dimglo}$-linearization of the very ample line bundle, that is induced from the canonical $\GL{\dimgra[\addshi]}$-linearization. As $\quotof{\twi{\sumto{\twish{\CY}}{\dimglo}}{-\casmu}}{\hilpo}\subseteq\grass{\dimgra[\addshi]-\hilpo(\addshi)}{\dimgra[\addshi]}$ is $\GL{\dimglo}$-invariant, the induced very ample line bundle on $\quotof{\twi{\sumto{\twish{\CY}}{\dimglo}}{-\casmu}}{\hilpo}$ is $\GL{\dimglo}$-linearized (e.g.\@ \cite{HuyLehn} p.\@ 101).

\smallskip

As it is explained in \cite{DerQuot}, choosing $\ribo>\lebo$ large enough we can realize $\quotof{\twi{\sumto{\twish{\CY}}{\dimglo}}{-\casmu}}{\hilpo}$ as the classical locus in a derived scheme $$\derqof{\twi{\sumto{\twish{\CY}}{\dimglo}}{-\casmu}}{\hilpo}{\lebo,\ribo},$$ that is fibered over
	\begin{equation}\derquz[\lebo,\ribo]:=\underset{\lebo\leq\addshi\leq\ribo}\prod\grass{\dimgra[\addshi]-\hilpo(\addshi)}{\dimgra[\addshi]},\end{equation}
which is a projective variety through the composite morphism
	\begin{equation}\label{allpro}\derquz[\lebo,\ribo]\hooklongrightarrow\underset{\lebo\leq\addshi\leq\ribo}\prod\mathbb P^{\dimpro{\addshi}}\hooklongrightarrow\mathbb P^{\underset{\lebo\leq\addshi\leq\ribo}\prod\lp\dimpro{\addshi}+1\rp-1},\end{equation}	
with the right arrow being a Segre embedding. Here is the construction of this derived scheme. For each $\addshi\in[\lebo,\ribo]$ let $\tabu[\addshi]$ be the tautological sub-bundle of rank $\dimgra[\addshi]-\hilpo(\addshi)$ of the trivial bundle $\tribun{\addshi}$ of rank $\dimgra[\addshi]=\dim\lp\Gamma\lp\CY,\twi{\sumto{\twish{\CY}}{\dimglo}}{\addshi-\casmu}\rp\rp$ on $\grass{\dimgra[\addshi]-\hilpo(\addshi)}{\dimgra[\addshi]}$. Pulling back to $\derquz[\lebo,\ribo]$ we denote
	\begin{equation}\tabu[\lebo,\ribo]:=\underset{\lebo\leq\addshi\leq\ribo}\bigboxplus\tabu[\addshi],\quad\tribun{\lebo,\ribo}:=\underset{\lebo\leq\addshi\leq\ribo}\bigboxplus\tribun{\addshi}.\end{equation}
Let $\gring[1,\ribo-\lebo]:=\underset{1\leq j\leq\ribo-\lebo}\bigoplus\gring[j]$, $\gring[j]:=\Gamma\lp \CY,\twish[j]{\CY}\rp$. We are interested in the following bundles on $\derquz[\lebo,\ribo]$:
\begin{equation}\label{dglie}\lf\Homo{\gring[1,\ribo-\lebo]^{\otimes^\alpha}\otimes\tabu[\lebo,\ribo]}{\tabu[\lebo,\ribo]}\rf_{\alpha\geq 1},\quad
	\lf\Homo{\gring[1,\ribo-\lebo]^{\otimes^{\alpha-1}}\otimes\tabu[\lebo,\ribo]}{\tribun{\lebo,\ribo}}\rf_{\alpha\geq 1},\end{equation}
where $\Homo{-}{-}$ stands for the bundle of \emph{homogeneous} morphisms (we view each $\gring[j]$, $\tabu[j]$ and $\tribun{j}$ as having homogeneous degree $j$). The total space of the bundle $\Homo{\gring[1,\ribo-\lebo]\otimes\tabu[\lebo,\ribo]}{\tabu[\lebo,\ribo]}$ will be $\derqof[0]{\twi{\sumto{\twish{\CY}}{\dimglo}}{-\casmu}}{\hilpo}{\lebo,\ribo}$, i.e.\@ the (cohomological) degree $0$ part of $\derqof{\twi{\sumto{\twish{\CY}}{\dimglo}}{-\casmu}}{\hilpo}{\lebo,\ribo}$. To construct the rest of this derived scheme, and especially the differential, we need to consider the $\twish{\derquz[\lebo,\ribo]}$-linear dg-Lie algebra structure on
\begin{equation}\label{biglie}\biglie:=
\underset{\alpha\geq 1}\bigoplus\biglie[\alpha],\,\,\, \biglie[\alpha]:=\Homo{\gring[1,\ribo-\lebo]^{\otimes^\alpha}\otimes\tabu[\lebo,\ribo]}{\tabu[\lebo,\ribo]}\oplus
	\Homo{\gring[1,\ribo-\lebo]^{\otimes^{\alpha-1}}\otimes\tabu[\lebo,\ribo]}{\tribun{\lebo,\ribo}}.\end{equation}
We assign \emph{cohomological} degree $\alpha$ to $\biglie[\alpha]$, the bracket is given by commutator of the composition, and the differential is given by commutator with the multiplication on $\gring[1,\ribo-\lebo]$ and with the $\gring[1,\ribo-\lebo]$-module structure on $\tribun{\lebo,\ribo}$ (e.g.\@ \cite{DerQuot} \S3.4, \cite{Ainfty} \S4.2). Applying the bar-construction to $\biglie$ we obtain a bundle of differential non-negatively graded co-commutative co-algebras on $\derquz[\lebo,\ribo]$ (e.g.\@ \cite{AlgOperads}  \S11). As this construction involves a shift down by $1$ in the cohomological degree, each $\biglie[\alpha]$ is re-assigned to degree $\alpha-1$. After taking the degree-wise linear dual over $\twish{\derquz[\lebo,\ribo]}$ this bundle of dg co-algebras becomes a bundle of differential non-positively graded commutative algebras $\bigring{\lebo,\ribo}$ (\cite{DerQuot} \S3.5, \cite{BKS} p.\@ 14). 

The spectrum of $\bigring{\lebo,\ribo}$ is a dg scheme fibered over $\derquz[\lebo,\ribo]$, but it is not yet $\derqof{\twi{\sumto{\twish{\CY}}{\dimglo}}{-\casmu}}{\hilpo}{\lebo,\ribo}$ that we want. The reason is that $\bigring{\lebo,\ribo}$ parametrizes all possible $\ainty$-module structures on $\tabu[\lebo,\ribo]$ over $\gring[1,\ribo-\lebo]$ and compatible $\ainty$-morphisms $\tabu[\lebo,\ribo]\rightarrow\tribun{\lebo,\ribo}$, there is no requirement for these morphisms to extend the canonical inclusion $\tabu[\lebo,\ribo]\hookrightarrow\tribun{\lebo,\ribo}$. To impose this requirement we need to consider the corresponding section $\derquz[\lebo,\ribo]\rightarrow\Homo{\tabu[\lebo,\ribo]}{\tribun{\lebo,\ribo}}$. This section defines an inclusion of the total spaces of bundles
\begin{equation}\label{inclusion}\Homo{\gring[1,\ribo-\lebo]\otimes\tabu[\lebo,\ribo]}{\tabu[\lebo,\ribo]}\hooklongrightarrow\Homo{\gring[1,\ribo-\lebo]\otimes\tabu[\lebo,\ribo]}{\tabu[\lebo,\ribo]}\underset{\derquz[\lebo,\ribo]}\times\Homo{\tabu[\lebo,\ribo]}{\tribun{\lebo,\ribo}}.\end{equation}
Notice that this inclusion is not linear, as the $0$-section is not mapped to the $0$-section (unless $\tabu[\lebo,\ribo]$ is the $0$-subbundle of $\tribun{\lebo,\ribo}$). Nonetheless, the image of (\ref{inclusion}) is a sub-scheme of the degree $0$ part of the spectrum of $\bigring{\lebo,\ribo}$. Restricting $\bigring{\lebo,\ribo}$ to this sub-scheme (i.e.\@ pulling back the dg structure sheaf to $\Homo{\gring[1,\ribo-\lebo]\otimes\tabu[\lebo,\ribo]}{\tabu[\lebo,\ribo]}$ over (\ref{inclusion})) we obtain another dg scheme fibered over $\derquz[\lebo,\ribo]$, which is $\derqof{\twi{\sumto{\twish{\CY}}{\dimglo}}{-\casmu}}{\hilpo}{\lebo,\ribo}$, the derived Quot-scheme that we want.\footnote{After this paper appeared on arXiv, the authors of \cite{DerQuot} have communicated to us that the main result of \cite{DerQuot} has a fundamental flaw. While fixing this flaw cannot be done within this paper, there is a straightforward way to do it, resulting in a dg scheme with the same properties as claimed by \cite{DerQuot}. This will be pursued elsewhere.}

\bigskip
	
Having constructed $\derqof{\twi{\sumto{\twish{\CY}}{\dimglo}}{-\casmu}}{\hilpo}{\lebo,\ribo}$ we would like to look at the action of $\GL{\dimglo}$ on it. There are canonical $\GL{\dimgra[\addshi]}$-linearizations of $\tabu[\addshi]$, $\tribun{\addshi}$ for each $\addshi\in[\lebo,\ribo]$. Defining 
	\begin{equation}\G[\lebo,\ribo]:=\underset{\lebo\leq\addshi\leq\ribo}\prod\GL{\dimgra[\addshi]}\end{equation}
we have the obvious action of $\G[\lebo,\ribo]$ on $\derquz[\lebo,\ribo]$ and $\G[\lebo,\ribo]$-linearizations of $\tabu[\lebo,\ribo]$, $\tribun{\lebo,\ribo}$. The very ample line bundle $\twish[1]{\derquz[\lebo,\ribo]}$ obtained from (\ref{allpro}) is the tensor product of line bundles pulled back from the factors, thus we have a $\G[\lebo,\ribo]$-linearization of $\twish[1]{\derquz[\lebo,\ribo]}$. Combining $\lf \GL{\dimglo}\hookrightarrow\GL{\dimgra_\addshi}\;|\; \addshi\in[\lebo,\ribo]\rf$ into one $\GL{\dimglo}\hookrightarrow\G[\lebo,\ribo]$ we have $\GL{\dimglo}$ acting on $\derquz[\lebo,\ribo]$ and $\GL{\dimglo}$-linearizations of $\tabu[\lebo,\ribo]$, $\tribun{\lebo,\ribo}$, $\twish[1]{\derquz[\lebo,\ribo]}$. This gives us, in a canonical way, $\GL{\dimglo}$-linearizations of various other bundles on $\derquz[\lebo,\ribo]$, obtained from $\tabu[\lebo,\ribo]$, $\tribun{\lebo,\ribo}$, as well as actions of $\GL{\dimglo}$ on the total spaces of these bundles.

In particular we have an action of $\GL{\dimglo}$ on the bundle of graded commutative algebras, that underlies $\bigring{\lebo,\ribo}$, as this bundle consists of tensor products of duals of (\ref{dglie}). We would like to see if this action of $\GL{\dimglo}$ is compatible with the differential on $\bigring{\lebo,\ribo}$. It is enough to check compatibility for each closed point $\matri\in\GL{\dimglo}$, and since the bar-construction is a functor, we can equivalently check that the action of $\matri$ is compatible with the bracket and the differential on (\ref{biglie}). 

Compatibility with the bracket is obvious, since the the action of $\matri$ is by conjugation and the bracket is the commutator of the composition of morphisms. The differential on $\biglie$ is a sum of two parts: the first one is given by the multiplication on $\gring[1,\ribo-\lebo]$ and the second by the $\gring[1,\ribo-\lebo]$-module structure on $\tribun{\lebo,\ribo}$. Since the action of $\GL{\dimglo}$ on $\gring[1,\ribo-\lebo]$ is trivial, clearly this action commutes with the first summand. Commutativity with the second summand follows from $\gring[1,\ribo-\lebo]$-linearity of the action of $\GL{\dimglo}$ on $\tribun{\lebo,\ribo}$.\footnote{Notice that the action of $\G[\lebo,\ribo]$ on $\tribun{\lebo,\ribo}$ is not $\gring[1,\ribo-\lebo]$-linear, in general.}

\smallskip

It is easy to see that $\derqof{\twi{\sumto{\twish{\CY}}{\dimglo}}{-\casmu}}{\hilpo}{\lebo,\ribo}$ is invariant with respect to the action of $\GL{\dimglo}$ on the spectrum of $\bigring{\lebo,\ribo}$. Indeed, the section $\derquz[\lebo,\ribo]\rightarrow\Homo{\tabu[\lebo,\ribo]}{\tribun{\lebo,\ribo}}$, we used to define (\ref{inclusion}), is invariant with respect to the action of $\GL{\dimglo}$ on $\derquz[\lebo,\ribo]$ and on the total space of $\Homo{\tabu[\lebo,\ribo]}{\tribun{\lebo,\ribo}}$, i.e.\@ the image of (\ref{inclusion}) is $\GL{\dimglo}$-invariant. Therefore we have an induced $\GL{\dimglo}$-action on $\derqof{\twi{\sumto{\twish{\CY}}{\dimglo}}{-\casmu}}{\hilpo}{\lebo,\ribo}$, and the projection 
	\begin{equation}\label{motogra}\derqof{\twi{\sumto{\twish{\CY}}{\dimglo}}{-\casmu}}{\hilpo}{\lebo,\ribo}\longrightarrow\derquz[\lebo,\ribo]\end{equation} 
is $\GL{\dimglo}$-invariant. Recall that the scalar matrices in $\GL{\dimglo}$ act as scalar matrices on $\tabu[\lebo,\ribo]$ and on $\tribun{\lebo,\ribo}$. Therefore, since the action of $\GL{\dimglo}$ on $\derqof{\twi{\sumto{\twish{\CY}}{\dimglo}}{-\casmu}}{\hilpo}{\lebo,\ribo}$ is given by conjugation, this action factors through $\PGL{\dimglo}$, just like the $\GL{\dimglo}$-action on $\derquz[\lebo,\ribo]$.

\bigskip

Now we look at the stable locus in $\derquz[\lebo,\ribo]$, and its pre-image in the derived Quot scheme $\derqof{\twi{\sumto{\twish{\CY}}{\dimglo}}{-\casmu}}{\hilpo}{\lebo,\ribo}$. This requires us to restrict our attention temporarily to the action of $\SL{\dimglo}$. Notice that both $\derquz[\lebo,\ribo]$ and  $\derqof[0]{\twi{\sumto{\twish{\CY}}{\dimglo}}{-\casmu}}{\hilpo}{\lebo,\ribo}$ are reduced and smooth. So, as in \cite{Newstead}, it makes sense to simplify our treatment and view these schemes as varieties, i.e.\@ to consider only the closed points.

The very ample line bundle on $\derquz[\lebo,\ribo]$, obtained from (\ref{allpro}), and its $\SL{\dimglo}$-linearization give us the GIT semi-stable locus $\semiquz[\lebo,\ribo]\subseteq\derquz[\lebo,\ribo]$, which contains the product of semi-stable loci $\semiquz[\addshi]\subseteq\derquz[\addshi]$, $\addshi\in[\lebo,\ribo]$. Within each $\semiquz[\addshi]$ we have the principal Luna stratum $\staquz[\addshi]$ consisting of $\pt_\addshi\in\semiquz[\addshi]$, for which we can choose an $r\in\mathbb N$ and an $\SL{\dimglo}$-invariant $\secti_\addshi\in\Gamma\lp\derquz[\addshi],\twish[r]{\derquz[\addshi]}\rp$ that does not vanish at $\pt_\addshi$, s.t.\@ the evaluation map $\SL{\dimglo}\rightarrow\lf\secti_{\addshi}\neq 0\rf\subseteq\derquz[\addshi]$ at $\pt_\addshi\in\derquz[\addshi]$ is proper, and, moreover, the stabilizer of $\pt_\addshi$ is the smallest possible, which in this case is $\mathbb Z_\dimglo\subset\SL{\dimglo}$. Then clearly $\pt:=\underset{\addshi\in[\lebo,\ribo]}\prod\pt_{\addshi}$ belongs to the principal Luna stratum $\staquz[\lebo,\ribo]\subseteq\semiquz[\lebo,\ribo]$, which is a Zariski open sub-variety. Let 
	\begin{equation}\staqof{\twi{\sumto{\twish{\CY}}{\dimglo}}{-\casmu}}{\hilpo}{\lebo,\ribo}\subseteq\derqof{\twi{\sumto{\twish{\CY}}{\dimglo}}{-\casmu}}{\hilpo}{\lebo,\ribo}\end{equation}
be the pre-image of $\staquz[\lebo,\ribo]$ with respect to (\ref{motogra}). It is obvious that the actions of $\PGL{\dimglo}$ on $\staqof{\twi{\sumto{\twish{\CY}}{\dimglo}}{-\casmu}}{\hilpo}{\lebo,\ribo}$, $\staquz[\lebo,\ribo]$ are free. As $\staquz[\lebo,\ribo]$ lies within the semi-stable locus of $\derquz[\lebo,\ribo]$ we have a good quotient $\staquz[\lebo,\ribo]\rightarrow\quo{\staquz[\lebo,\ribo]}{\PGL{\dimglo}}$. Since (\ref{motogra}) is affine, i.e.\@ there is an affine atlas on $\derquz[\lebo,\ribo]$, s.t.\@ pre-image of each chart is given by one dg algebra, we have a good quotient 					
	\begin{equation}\staqof{\twi{\sumto{\twish{\CY}}{\dimglo}}{-\casmu}}{\hilpo}{\lebo,\ribo}\rightarrow\quo{\staqof{\twi{\sumto{\twish{\CY}}{\dimglo}}{-\casmu}}{\hilpo}{\lebo,\ribo}}{\PGL{\dimglo}}\end{equation}
as well. An easy way to see it is to rewrite $\staqof{\twi{\sumto{\twish{\CY}}{\dimglo}}{-\casmu}}{\hilpo}{\lebo,\ribo}$ as spectrum of a  simplicial algebra, i.e.\@ as a co-simplicial diagram of affine varieties over $\staquz[\lebo,\ribo]$, and then use functoriality of good quotients with respect to affine morphisms (e.g.\@ \cite{Newstead} Prop.\@ 3.12, p.\@ 58).

It is known (e.g.\@ \cite{HuyLehn} Cor.\@ 4.3.5, p.\@ 102) that $\staquz[\lebo,\ribo]\rightarrow\quo{\staquz[\lebo,\ribo]}{\PGL{\dimglo}}$ is a principal $\PGL{\dimglo}$-bundle in the \'etale topology. The following proposition is a straightforward extension of this statement to $\staqof{\twi{\sumto{\twish{\CY}}{\dimglo}}{-\casmu}}{\hilpo}{\lebo,\ribo}$. Given an \'etale chart $U\rightarrow\quo{\staquz[\lebo,\ribo]}{\PGL{\dimglo}}$ we will denote by $$\lp\staqof{\twi{\sumto{\twish{\CY}}{\dimglo}}{-\casmu}}{\hilpo}{\lebo,\ribo}\rp_U$$ and$$\lp\quo{\staqof{\twi{\sumto{\twish{\CY}}{\dimglo}}{-\casmu}}{\hilpo}{\lebo,\ribo}}{\PGL{\dimglo}}\rp_U$$ the corresponding pullbacks.
\begin{proposition}\label{LocalProduct} We have
		\begin{align}&\staqof{\twi{\sumto{\twish{\CY}}{\dimglo}}{-\casmu}}{\hilpo}{\lebo,\ribo}\cong\notag\\
		&\staquz[\lebo,\ribo]\underset{\quo{\staquz[\lebo,\ribo]}{\PGL{\dimglo}}}\times\lp\quo{\staqof{\twi{\sumto{\twish{\CY}}{\dimglo}}{-\casmu}}{\hilpo}{\lebo,\ribo}}{\PGL{\dimglo}}\rp,\end{align}
in particular, for any $\pt\in\quo{\staquz[\lebo,\ribo]}{\PGL{\dimglo}}$ there is an \'etale  chart $U\rightarrow\quo{\staquz[\lebo,\ribo]}{\PGL{\dimglo}}$, such that 
		\begin{align}&\lp\staqof{\twi{\sumto{\twish{\CY}}{\dimglo}}{-\casmu}}{\hilpo}{\lebo,\ribo}\rp_U\cong\notag\\
		& \PGL{\dimglo}\times\lp\quo{\staqof{\twi{\sumto{\twish{\CY}}{\dimglo}}{-\casmu}}{\hilpo}{\lebo,\ribo}}{\PGL{\dimglo}}\rp_U,\end{align}
and the action of $\PGL{\dimglo}$ is given by its left action on itself.\end{proposition}
\begin{proof} To construct $\derqof{\twi{\sumto{\twish{\CY}}{\dimglo}}{-\casmu}}{\hilpo}{\lebo,\ribo}$ we began with the bundles (\ref{biglie}) on $\derquz[\lebo,\ribo]$, then proceed by taking their sums, $\twish{\derquz[\lebo,\ribo]}$-linear duals and tensor products, and eventually quotients. It is important to notice that (\ref{biglie}) carry a $\PGL{\dimglo}$-linearization, not just a $\GL{\dimglo}$-linearization, therefore all these bundles and the $\PGL{\dimglo}$-equivariant morphisms between them descend to $\quo{\staquz[\lebo,\ribo]}{\PGL{\dimglo}}$ (e.g.\@ \cite{HuyLehn} Thm.\@ 4.2.14, p.\@ 98).\end{proof}

\section{Lagrangian distributions on quotient stacks}

In this section we analyze shifted symplectic structures and, more generally, homotopically closed forms on quotient stacks of derived schemes with respect to actions by linearly reductive groups. To be able to deal with these structures we need to have an efficient way of presenting homotopically correct cotangent complexes on affine derived schemes. Following \cite{BSY} Def.\@ 1 we will assume that for any differential non-positively graded $\mathbb C$-algebra $\dera{R}$ that we consider, the $\mathbb C$-algebra $\dera[0]{R}$ is finitely generated and smooth, and the underlying graded algebra $\dera[*]{R}$ is freely generated over $\dera[0]{R}$ by a finite sequence of finitely generated projective modules. Such algebras are cofibrant enough for the usual K\"ahler differentials to compute the correct cotangent complexes (e.g.\@ \cite{BSY} Prop.\@ 1 p.\@ 6). Affine derived schemes $\spec[\dera{R}]$ defined by such algebras will be called \emph{affine dg manifolds}.

In general a derived $\quot$-scheme is not an affine dg manifold because the degree $0$ component is only quasi-projective. We will use the name \emph{dg manifolds} to denote smooth quasi-projective schemes, equipped with extensions of the structure sheaves to sheaves of differential non-positively graded $\mathbb C$-algebras, s.t.\@ the underlying sheaves of graded algebras are freely generated by finite sequences of locally free coherent sheaves in negative degrees. Thus every affine dg manifold is a dg manifold and every dg manifold is Zariski locally an affine dg manifold. This notion is a slight variation of the one in \cite{DerQuot} Def.\@ 2.5.1 p.\@ 415, where the conditions of being quasi-projective in degree $0$ and having only a finite sequence of generating bundles were not imposed. As in \cite{BKS} we view each dg manifold as a derived scheme by breaking it into a simplicial diagram of affine derived schemes.

\smallskip

 As our intended objects of study are quotient stacks of $\quot$-schemes, we will follow Geometric Invariant Theory and assume that, when we consider a linear algebraic group $\algr$ acting on a derived scheme $\dersch$, e.g.\@ given by a dg manifold, there is an atlas on $\dersch$ consisting of $\algr$-invariant affine derived schemes. Thus locally we will have $\algr$ acting on some affine dg manifold $\spec[\dera{R}]$. 

\subsection{Shifted symplectic structures on quotient stacks}\label{twoone}

To fix the notation, we begin with recalling the notion of homotopically closed differential forms on affine derived schemes (\cite{PTVV13} Def.\@ 1.8 p.\@ 290). For a dg algebra $\dera{R}$ we have the cotangent complex $\derham{\dera{R}}$ concentrated in non-positive degrees. Denoting by $\suspense{\derham{\dera{R}}}$ the suspension of $\derham{\dera{R}}$, i.e.\@ $\suspense{\derham{\dera{R}}}:=\shift[-1]{\dera{R}}\underset{\dera{R}}\otimes\derham{\dera{R}}$, where $\shift[-1]{\dera{R}}$ is the free $\dera{R}$ module on a single generator in degree $-1$,\hide{\footnote{Notice that when we apply the differential from the left on $\suspense{\derham{\dera{R}}}$ it has to pass over the generator in degree $-1$, resulting in the differential on $\suspense{\derham{\dera{R}}}$ being $-1$ times the differential on $\derham{\dera{R}}$. This accounts for the natural sign trick producing a graded mixed complex, rather than a complex of complexes.}} we form the graded mixed algebra
\begin{equation}\label{derhamalgebra}\derham[\bullet]{\dera{R}}:=\lp\underset{i\in\noneg}\bigoplus\;\Sym[\dera{R}]{i}{\suspense{\derham{\dera{R}}}},\homdi,\dedi\rp,\end{equation}
where $\homdi$, $\dedi$ are the cohomological and de Rham differentials, having cohomological degrees $1,-1$ and weights $0,1$ respectively. Clearly $\derham[\bullet]{\dera{R}}$ is concentrated in non-positive degrees ($\dera{R}$ sits in degree $0$) and in non-negative weights (an $i$-form has weight $i$). To define homotopically closed $2$-forms on $\spec[\dera{R}]$ we need the negative cyclic complex of weight $2$ corresponding to $\derham[\bullet]{\dera{R}}$:
\begin{equation}\negacy{\bullet}\lp\derham[\bullet]{\dera{R}}\rp:=\lp\underset{j\in\mathbb Z}\bigoplus\;\negacy{j}\lp\derham[\bullet]{\dera{R}}\rp,\homdi+\para\dedi\rp,\quad\negacy{j}\lp\derham[\bullet]{\dera{R}}\rp:=\lf\underset{i\geq 0}\sum\para^i\nform{2+i}\rf,\end{equation}
where $\para$ is a formal parameter of degree $2$ and weight $-1$, $\nform{2+i}$ stands for a $2+i$-form of degree $j-2i$. All elements of $\negacy{\bullet}$ have the same weight $2$, i.e.\@ this complex has only one grading (cohomological). 

\begin{remark} We should note that when we say $\nform{i+2}$ has degree $j-2i$ we mean the cohomological degree within $\derham[\bullet]{\dera{R}}$, which involves the suspension, i.e.\@ shifting down by the weight of the form. For example in our notation a $2$-form of degree $j$ defines a morphism of degree $j+2$ from the tangent complex to the cotangent complex on $\spec[\dera{R}]$.\end{remark}

For any $\shi\in\mathbb Z$, truncating $\suspense[\shi-2]{\negacy{\bullet}\lp\derham[\bullet]{\dera{R}}\rp}=\shift[2-\shi]{\mathbb C}\underset{\mathbb C}\otimes\negacy{\bullet}\lp\derham[\bullet]{\dera{R}}\rp$ to non-positive degrees\footnote{Recall that truncating a cochain complex to some degree and below means keeping only the cocycles in that degree and everything of smaller degree.} we obtain a cochain complex $\hoclo{\shi}\lp\spec[\dera{R}]\rp$ in the category $\vect$ of vector spaces over $\mathbb C$. This cochain complex is contravariantly functorial in $\spec[\dera{R}]$, and defines a stack -- \emph{the stack of homotopically closed $2$-forms of degree $\shi$} -- on the site $\saff$ of affine derived schemes (\cite{PTVV13} Def.\@ 1.8, p.\@ 290, Prop.\@ 1.11 p.\@ 291). 

\begin{remark}\label{doldkan} Note that $\hoclo{\shi}$ is a (homotopy) sheaf of simplicial sets on the site of affine derived schemes, that is obtained by applying objectwise Dold--Kan correspondence to a sheaf of non-positively graded cochain complexes in $\vect$, i.e.\@ $\hoclo{\shi}$ is a sheaf of simplicial objects in $\vect$.\end{remark}

\emph{A homotopically closed $2$-form of degree $\shi$} on $\spec[\dera{R}]$ is defined as a $0$-simplex in $\hoclo{\shi}\lp\spec[\dera{R}]\rp$. As the truncation to non-positive degrees involves taking the cocycles in degree $0$, to choose such a simplex is the same as to choose a $\homdi+\para\dedi$-cocycle of degree $\shi-2$ in $\negacy{\bullet}\lp\derham[\bullet]{\dera{R}}\rp$. One should notice that the definition of $\hoclo{\shi}$ above depends on our choice of presentations of cotangent complexes on affine derived schemes, hence, in general, $\hoclo{\shi}$ is only defined up to a global (i.e.\@ an objectwise) weak equivalence. Therefore a $0$-simplex as above is usually given only up to a homotopy provided by $1$-simplices, i.e.\@ up to a $\homdi+\para\dedi$-coboundary.

\smallskip

To define a homotopically closed $2$-form of degree $\shi$ on some stack $\astack$ on $\saff$ means to give (a homotopy class of) a morphism of stacks $\astack\rightarrow\hoclo{\shi}$ (\cite{PTVV13} Def.\@ 1.12 p.\@ 292). We have an entire simplicial set $\map{\astack}{\hoclo{\shi}}$ of such morphisms and we would like to compute it in the particular case of $\astack$ being the \emph{quotient stack} $\quack{\spec[\dera{R}]}{\algr}$ for an action of a linear algebraic group $\algr$ on an affine dg manifold $\spec[\dera{R}]$. Using the action  and the group structure on $\algr$ we construct a simplicial diagram of affine dg manifolds $\duack{\spec[\dera{R}]}{\algr}:=\lf\spec[\dera{R}]\times\lp\algr^{\times^j}\rp\rf_{j\in\noneg}$ and then define
	\begin{equation}\quack{\spec[\dera{R}]}{\algr}:=\hocolim\duack{\spec[\dera{R}]}{\algr}\end{equation}
computed\hide{\footnote{This homotopy colimit can be computed by taking the realization of the sheaf of bisimplicial sets (e.g.\@ \cite{Hirsh} Def.\@ 18.6.2 p.\@ 395, Thm.\@ 18.7.4 p.\@ 397), i.e.\@ by taking the diagonals (e.g.\@ \cite{GJ} Ex.\@ 1.4 p.\@ 198).}} in the category of stacks, where we view each $\spec[\dera{R}]\times\lp\algr^{\times^j}\rp$ as a stack
	\begin{equation}\label{yoneda}\forall\spec[\dera{Q}]\in\saff\quad\spec[\dera{Q}]\longmapsto\map[\saff]{\spec[\dera{Q}]}{\spec[\dera{R}]\times\lp\algr^{\times^j}\rp}.\end{equation} 
Since construction of (\ref{yoneda}) involves taking the fibrant resolutions of derived affine schemes $\lf\spec[\dera{R}]\times\lp\algr^{\times^j}\rp\rf_{j\in\noneg}$ in $\saff$, we can assume that the simplicial diagram $\duack{\spec[\dera{R}]}{\algr}$  is objectwise cofibrant in a local model structure on the category of stacks. Therefore we have (e.g.\@ \cite{Hirsh} Thm.\@ 19.4.4 p.\@ 415):
	\begin{equation}\label{holim}\map{\quack{\spec[\dera{R}]}{\algr}}{\hoclo{\shi}}\simeq\holim[j\in\noneg]\map{\spec[\dera{R}]\times\lp\algr^{\times^j}\rp}{\hoclo{\shi}}\simeq\end{equation}
	\begin{equation*}\simeq\holim[j\in\noneg]\hoclo{\shi}\lp\spec[\dera{R}]\times\lp\algr^{\times^j}\rp\rp.\end{equation*}
Fibrant resolutions of $\spec[\dera{R}]$ and $\algr$ correspond to cofibrant resolutions of $\dera{R}$ and of the ring of functions on $\algr$, however, since it is $\hoclo{\shi}$ that we evaluate on $\spec[\dera{R}]\times\lp\algr^{\times^j}\rp$, here it is enough to require that K\"ahler differentials compute the correct cotangent complexes on $\algr$ and on $\spec[\dera{R}]$, i.e.\@ the assumption that $\spec[\dera{R}]$ is an affine dg manifold is sufficient. Thus we have the following

\begin{proposition} For any linear algebraic group $\algr$ acting on  $\spec[\dera{R}]$, and for any $\shi\in\mathbb Z$ the simplicial set of homotopically closed $2$-forms of degree $\shi$ on $\quack{\spec[\dera{R}]}{\algr}$ is weakly equivalent to a homotopy limit of the cosimplicial diagram $\lf\hoclo{\shi}\lp\spec[\dera{R}]\times\lp\algr^{\times^j}\rp\rp\rf_{j\in\noneg}$ of simplicial sets.\end{proposition}

Assuming that a cosimplicial diagram in a simplicial model category is fibrant in the Reedy model structure, one way to compute its homotopy limit is by taking the total space (e.g.\@ \cite{Hirsh} Thm.\@ 18.7.4 p.\@ 397). In turn (e.g.\@ \cite{GJ} p.\@ 389) computing the total space of $\lf\hoclo{\shi}\lp\spec[\dera{R}]\times\lp\algr^{\times^j}\rp\rp\rf_{j\in\noneg}$ is equivalent to computing the mapping space
	\begin{equation}\label{totmap}\map[\coss]{\stasi}{\lf\hoclo{\shi}\lp\spec[\dera{R}]\times\lp\algr^{\times^j}\rp\rp\rf_{j\in\noneg}},\end{equation}
where $\coss$ is the category of cosimplicial diagrams in the category $\sset$ of simplicial sets, and $\stasi=\lf\stasi[j]\rf_{j\in\noneg}$ is the cosimplicial diagram of the standard simplices in $\sset$. Here Rem.\@ \ref{doldkan} becomes useful. Denoting by $\coss[\mathbb C]$ the category of cosimplicial diagrams in the category $\sset[\mathbb C]$ of simplicial objects in $\vect$ and by $\stasic\in\coss[\mathbb C]$ the $\mathbb C$-linearization of $\stasi$ we can use the adjunction $\coss\leftrightarrows\coss[\mathbb C]$ to rewrite (\ref{totmap}) as
	\begin{equation}\label{totmac}\map[{\coss[\mathbb C]}]{\stasic}{\lf\hoclo{\shi}\lp\spec[\dera{R}]\times\lp\algr^{\times^j}\rp\rp\rf_{j\in\noneg}}.\end{equation}
The simplicial normalization functor (e.g.\@ \cite{DK4} \S3.1) gives an equivalence between $\sset[\mathbb C]$ and the category $\dgcon[\mathbb C]$ of non-positively graded cochain complexes in $\vect$. The latter is an abelian category, and composing with the cosimplicial normalization functor (e.g.\@ \cite{DK4} \S3.2) we have an equivalence
	\begin{equation}\cnorm\colon\coss[\mathbb C]\overset{\cong}\longrightarrow\dgcop[{\dgcon[\mathbb C]}]\cong\dgcon[{\dgcop[\mathbb C]}],\end{equation}
where $\dgcop$ denotes non-negatively graded cochain complexes. As $\dgcop[\mathbb C]$ is abelian there is the standard projective model structure on $\dgcon[{\dgcop[\mathbb C]}]$.\footnote{Weak equivalences are the \emph{row-wise} quasi-isomorphisms.} We have the following straightforward proposition
\begin{proposition}\label{cosimplicialnorm} The functor $\cnorm$ is a left Quillen functor with respect to the Reedy model structure on $\coss[\mathbb C]$ and the projective model structure on $\dgcon[{\dgcop[\mathbb C]}]$.\end{proposition}
\begin{proof} First we show that $\cnorm$ preserves cofibrations. In the Reedy model structure on $\coss[\mathbb C]$ each cofibration is in particular a cofibration in each cosimplicial dimension (e.g.\@ \cite{Hirsh} Prop.\@ 15.3.1 p.\@ 291). Since the normalization functor in the simplicial Dold--Kan correspondence is left Quillen, it follows that any cofibration in $\coss[\mathbb C]$ is mapped to a cosimplicial diagram of cofibrations in $\dgcon[\mathbb C]$, i.e.\@ to a cosimplicial diagram of injective maps. Then the cosimplicial normalization functor extracts subcomplexes in each cosimplicial dimension (e.g.\@ \cite{DK4} \S3.2), hence preserving injectivity. 
	
To prove that $\cnorm$ preserves weak equivalences we need to show that cosimplicial normalization maps cosimplicial diagrams of weak equivalences in $\dgcon[\mathbb C]$ to weak equivalences in $\dgcon[{\dgcop[\mathbb C]}]$. We note that the normalized complexes extracted by cosimplicial normalization have canonical complements -- the bulk complexes (loc.\@ cit.\@). Since cohomology of a direct sum of two complexes is a direct sum of cohomologies, we conclude that $\cnorm$ preserves weak equivalences.\end{proof}

\begin{remark} An immediate consequence of Prop.\@ \ref{cosimplicialnorm} is that every object in $\coss[\mathbb C]$ is fibrant, and hence every cosimplicial diagram of simplicial objects in $\vect$ is fibrant in the Reedy model structure on $\coss$. Therefore (\ref{totmap}) does compute the homotopy limit in (\ref{holim}).\end{remark}

Let $\ddcom[\mathbb C]{\bullet}{\bullet}$ be the category of (unbounded) double cochain complexes in $\vect$.\footnote{The $(i,j)$ component is in the $j$-th row and the $i$-th column.} There is a standard functor $\sig\colon\dgcon[{\dgcop[\mathbb C]}]\rightarrow\ddcom[\mathbb C]{\bullet}{\bullet}$,\hide{\footnote{This functor involves ``the sign trick'' (e.g.\@ \cite{Wei} p.\@ 8) multiplying the horizontal differential in row $j$ by $(-1)^j$.}} whose image is the full subcategory of second quadrant complexes. Let $\grmix[\mathbb C]{\bullet}{\bullet}$ be the category of (unbounded) graded mixed complexes equipped with the projective model structure (\cite{PTVV13} \S1.1). In our notation the $(i,j)$ component of $\grmix[\mathbb C]{\bullet}{\bullet}$  is of \emph{mixed degree} $i$ and \emph{weight} $j$. There is the obvious equivalence $\mixing\colon\ddcom[\mathbb C]{\bullet}{\bullet}\overset\cong\longrightarrow\grmix[\mathbb C]{\bullet}{\bullet}$ which renames an $(i,j)$-component into an $(i-j,j)$-component. Composing $\cnorm$, $\sig$ and $\mixing$ we obtain
	\begin{equation}\cmorm\colon\coss[\mathbb C]\longrightarrow\grmix[\mathbb C]{\bullet}{\bullet}.\end{equation}
\begin{proposition} The functor $\cmorm$ is a left Quillen functor with respect to the Reedy model structure on $\coss[\mathbb C]$ and the projective model structure on $\grmix[\mathbb C]{\bullet}{\bullet}$. This functor realizes $\coss[\mathbb C]$ as a full co-reflective subcategory of $\grmix[\mathbb C]{\bullet}{\bullet}$ consisting of objects whose components are $0$ unless they satisfy:
	\begin{equation}{\rm weight}\geq 0,\quad{\rm mixed\; degree}+{\rm weight}\leq 0.\end{equation}
\end{proposition}
\begin{proof} Let $\dgcou[{\dgcou[\mathbb C]{\bullet}}]{\bullet}$ be the category of unbounded cochain complexes in the category of unbounded cochain complexes in $\vect$. As $\dgcou[\mathbb C]{\bullet}$ is abelian we have the projective model structure on $\dgcou[{\dgcou[\mathbb C]{\bullet}}]{\bullet}$, that we transfer over the equivalence $\dgcou[{\dgcou[\mathbb C]{\bullet}}]{\bullet}\cong\ddcom[\mathbb C]{\bullet}{\bullet}$. It is clear that $\mixing$ identifies this model structure with the one on $\grmix[\mathbb C]{\bullet}{\bullet}$. According to Prop.\@ \ref{cosimplicialnorm} $\cnorm$ is a left Quillen functor, so we need to show that $\sig$ is also a left Quillen functor. 
	
The right adjoint of $\sig$ is the truncation to the second quadrant, i.e.\@ it consists of erasing everything of negative vertical degree and then truncating to the non-positive horizontal degrees. Clearly this operation preserves row-wise quasi-isomorphisms and it maps surjections to morphisms that are surjective in negative horizontal degrees, i.e.\@ the right adjoint is a right Quillen functor.\end{proof}

\begin{corollary}\label{simplicialadjunction} For any $S,T\in\coss[\mathbb C]$ we have
		\begin{equation}\map[{\coss[\mathbb C]}]{S}{T}\simeq\map[{\grmix[\mathbb C]{\bullet}{\bullet}}]{\cmorm\lp S\rp}{\cmorm\lp T\rp}.\end{equation}
\end{corollary}
\hide{\begin{proof} We can assume that $T$ is the image of a fibrant replacement $T'$ of $\cmorm\lp T\rp$ in $\grmix[\mathbb C]{\bullet}{\bullet}$. Then to compute the mapping space in $\coss[\mathbb C]$ we can choose a cosimplicial resolution $\widetilde{S}\rightarrow S$ (e.g.\@ \cite{Hirsh} Def.\@ 16.1.2 p.\@ 318) and take the resulting simplicial set $\hom_{\coss[\mathbb C]}\lp\widetilde{S},T\rp$ (e.g.\@ \cite{Hirsh} Def.\@ 17.1.1 p.\@ 349). Now we have $\hom_{\coss[\mathbb C]}\lp\widetilde{S},T\rp\cong\hom_{\grmix[\mathbb C]{\bullet}{\bullet}}\lp\cmorm\lp\widetilde{S}\rp,T'\rp\simeq\map[{\grmix[\mathbb C]{\bullet}{\bullet}}]{\cmorm\lp S\rp}{\cmorm\lp T\rp}$ (e.g.\@ \cite{Hirsh} Prop.\@ 17.4.16 p.\@ 357).\end{proof}}

Having the Corollary above we would like to understand the graded mixed complexes
	\begin{equation}\cmorm\lp\lf\hoclo{d}\lp\spec[\dera{R}]\times\lp\algr^{\times^j}\rp\rp\rf_{j\in\noneg}\rp,\quad\cmorm\lp\stasic\rp.\end{equation}
It is easy to describe the former: in weight $j\in\noneg$ it consists of elements of $\hoclo{d}\lp\spec[\dera{R}]\times\lp\algr^{\times^j}\rp\rp$ that pull back to $0$ over any degeneration map $\spec[\dera{R}]\times\lp\algr^{\times^{j-1}}\rp\hookrightarrow\spec[\dera{R}]\times\lp\algr^{\times^j}\rp$. 
\begin{proposition} The graded mixed complex $\cmorm\lp\stasic\rp$ is a cofibrant replacement of $\mathbb C$ considered as a graded mixed complex concentrated in degree $0$ and weight $0$.\end{proposition}
\begin{proof} For each $j\in\noneg$ simplicial normalization $\snorm\lp\stasic[j]\rp$ of $\stasic[j]$ is generated by the non-degenerate simplices of $\stasic[j]$. When we apply the cosimplicial normalization to $\lf\snorm\lp\stasic[j]\rp\rf_{j\in\noneg}$ we throw away the bulk subcomplex, which for each $j\geq 1$ consists of the sum of images of the co-face maps $\lf\codeg^r\colon\stasic[j-1]\hookrightarrow\stasic[j]\rf_{1\leq r\leq j}$ (e.g.\@ \cite{DK4} \S3.2\hide{, \cite{Good} \S2}). Notice the absence of $\codeg^0$. This means that all simplices in $\stasi[j]$ are hit, except for the non-degenerate simplex in dimension $j$ and one of its faces. Therefore $\cmorm\lp\stasic\rp$ has a basis over $\mathbb C$ consisting of $\lf a_j\rf_{j\geq 0}\cup\lf b_j\rf_{j\geq 1}$, where each $a_j$ has mixed degree $-2j$ and weight $j$, each $b_j$ has mixed degree $-2j+1$ and weight $j$, and $\homdi\lp a_j\rp=b_{j}$, $\dedi\lp a_j\rp=b_{j+1}$. This is exactly the complex $\mathcal Q(0)$ considered in the proof of Prop.\@ 1.3 in \cite{PTVV13}.\end{proof}

Now we look at computing mapping spaces in $\grmix[\mathbb C]{\bullet}{\bullet}$. As the model structure is based on the standard projective model structure of unbounded cochain complexes in $\vect$, we have unbounded cochain complexes of morphisms between objects in $\grmix[\mathbb C]{\bullet}{\bullet}$ (\cite{PTVV13} proof of Prop.\@ 1.3).\hide{ Taking cofibrant replacements, if necessary, these cochain complexes are obtained by taking maps to the target from all possible shifts of the source. Notice that the weight is not shifted, i.e.\@ all elements of the cochain complexes of morphisms have weight $0$. In other words we look at morphisms of dg modules in the category of weight-graded modules over the weight-graded unital ring generated by $\epsilon$ that has weight $1$ and squares to $0$. This also explains why the mapping spaces out of $\cmorm\lp\stasic\rp$ are exactly the product-total complexes.} Truncating these cochain complexes to non-positive degrees we obtain the (normalizations of) the mapping spaces. In other words for any $T\in\grmix[\mathbb C]{\bullet}{\bullet}$ we have (loc.\@ cit.\@)
	\begin{equation}\map[{\grmix[\mathbb C]{\bullet}{\bullet}}]{\cmorm\lp\stasic\rp}{T}\cong\negacy[0]{\leq 0}\lp T\rp,\end{equation}
where $\negacy[0]{\leq 0}$ means truncation to non-positive degrees of the negative cyclic complex of weight $0$. Applying this to (\ref{totmac}) we obtain an explicit description of the space of homotopically closed $2$-forms of degree $\shi$ on a quotient stack as follows.

\begin{proposition}\label{allgroups} The normalization of the simplicial set of homotopically closed $2$-forms of degree $\shi$ on $\quack{\spec[\dera{R}]}{\algr}$ is weakly equivalent to
		\begin{equation}\label{formstack}\staclo{\shi}\lp\quack{\spec[\dera{R}]}{\algr}\rp=\underset{r\leq 0}\bigoplus\;\staclo[r]{\shi}\lp\quack{\spec[\dera{R}]}{\algr}\rp,\end{equation}
	\begin{equation}\label{series}\staclo[r]{\shi}\lp\quack{\spec[\dera{R}]}{\algr}\rp:=\lf\underset{i,j\geq 0}\sum\sara^j\para^i\nform[j]{2+i}\rf,\end{equation}
where $\nform[j]{2+i}$ is a $2+i$-form on $\spec[\dera{R}]\times\lp\algr^{\times^j}\rp$ of degree $r+\shi-2-2i-2j$. The differential on $\staclo{\shi}\lp\quack{\spec[\dera{R}]}{\algr}\rp$ is $\homdi+\para\dedi+\sara\stadi$, where $\stadi$ is the alternating sum of all pullbacks over face maps in $\lf\spec[\dera{R}]\times\lp\algr^{\times^j}\rp\rf_{j\in\noneg}$. Elements of $\staclo[0]{\shi}\lp\quack{\spec[\dera{R}]}{\algr}\rp$ have to be $\homdi+\para\dedi+\sara\stadi$-cocycles.
				
Any homotopically closed $2$-form of degree $\shi$ on $\quack{\spec[\dera{R}]}{\algr}$ can be described as an element of  $\staclo[0]{\shi}\lp\quack{\spec[\dera{R}]}{\algr}\rp$. Two such forms are equivalent if they differ by a $\homdi+\para\dedi+\sara\stadi$-coboundary.\end{proposition}
\begin{proof} From the proof of Prop.\@ 1.3 in \cite{PTVV13} we know that (\ref{formstack}) is the space of maps from $\cmorm\lp\stasic\rp$ to graded mixed complex of \emph{all} forms on $\lf\spec[\dera{R}]\times\lp\algr^{\times^j}\rp\rf_{j\geq 0}$. The only reason this Proposition does not immediately follow from Corollary \ref{simplicialadjunction} is that the space 
		\begin{equation}\label{truespace}\map[{\grmix[\mathbb C]{\bullet}{\bullet}}]{\cmorm\lp\stasic\rp}{\cmorm\lp\lf\hoclo{\shi}\lp\spec[\dera{R}]\times\lp\algr^{\times^j}\rp\rp\rf_{j\in\noneg}\rp}\end{equation} 
is not all of (\ref{formstack}) because normalization does not consist of all forms. In (\ref{truespace}) there is an additional condition that $\forall i,j$ pullbacks of $\nform[j]{2+i}$ over all degeneration maps in $\lf\spec[\dera{R}]\times\lp\algr^{\times^j}\rp\rf_{j\in\noneg}$ are $0$. In other words (\ref{formstack}) contains also the mapping space from $\cmorm\lp\stasic\rp$ to the bulk complex in cosimplicial normalization. However, this bulk double complex is acyclic in the vertical direction (e.g.\@ \cite{DK4} \S3.2) and the Acyclic Assembly Lemma (e.g.\@ \cite{Wei} Lemma 2.7.3 p.\@ 59) tells us that the corresponding product-total complex is acyclic. We observe that this product-total complex is exactly the mapping space from $\cmorm\lp\stasic\rp$ to the bulk double complex.\end{proof}

If we want to describe homotopically closed forms on $\quack{\dgsch}{\algr}$, where $\dgsch$ is a dg manifold that is not necessarily affine, we can represent $\quack{\dgsch}{\algr}$ as a simplicial diagram of quotients of affine dg manifolds by actions of $\algr$. The entire process can be repeated and we would obtain expressions as in (\ref{series}) just with an additional formal parameter keeping track of the number of affine charts that intersect. The differential will have an additional summand as well.

\begin{remark}\label{hypercohomology} Assuming that $\dgsch[0]$ has a good quotient with respect to the action of $\algr$\hide{\footnote{As $\dgsch\rightarrow\dgsch[0]$ is a (graded) affine morphism, existence of a good quotient in degree $0$ implies the same in all degrees.}} there is an equivalent way to describe the space of homotopically closed forms on $\quack{\dgsch}{\algr}$.  We can use $\staclo{\shi}\lp\quack{\spec[\dera{R}]}{\algr}\rp$ to construct a sheaf $\staclo{\shi}$ of $\mathbb C$-linear cochain complexes on the topological space $\quospa$ underlying $\quo{\dgsch}{\algr}$. Then $\hyperco[0]\lp\quospa,\staclo{\shi}\rp$ is the vector space of equivalence classes of homotopically closed $2$-forms of degree $\shi$ on $\quack{\dgsch}{\algr}$. One can compute hypercohomology by using \v{C}ech covers consisting of affine charts, as it was done for derived schemes in \cite{BSY} Def.\@ 10.\end{remark}

\bigskip

In addition to being homotopically closed, a shifted symplectic structure $\underset{i,j\geq 0}\sum\sara^j\para^i\nform[j]{2+i}$ has to be \emph{non-degenerate}, i.e.\@ its $\para$-free term $\underset{j\geq 0}\sum\sara^j\nform[j]{2}$ has to define a weak equivalence between the tangent complex suspended $-\shi$ times and the cotangent complex. These complexes are defined for points in stacks, and in the case of $\quack{\spec[\dera{R}]}{\algr}$ it is enough to look at just one point: the canonical $\spec[\dera{R}]\rightarrow\quack{\spec[\dera{R}]}{\algr}$. The cotangent complex for this point is the total complex of the bi-complex
\begin{equation}\label{bicomplex}\derham{{\dera{R}}}\rightarrow\allie^*,\end{equation}
where $\allie^*$ is the trivial bundle on $\spec[{\dera{R}}]$ with the fiber being the $\mathbb C$-linear dual of the Lie algebra of $\algr$. In this situation the free term $\nform[0]{2}$ on $\spec[\dera{R}]$ is made of sections of $\derham{\dera{R}}$, while restriction of $\nform[1]{2}$ to the degeneration map $\spec[\dera{R}]\hookrightarrow\spec[\dera{R}]\times\algr$ involves sections of $\derham{\dera{R}}$ and those of $\allie^*$. It is the latter term that pairs ${\rm Ext}^0$ and ${\rm Ext}^{-\shi+2}$ of sheaves on a Calabi--Yau $2-\shi$-fold.

\subsection{Invariant symplectic structures and Lagrangian distributions}\label{twotwo}

The simplicial set 
	\begin{equation}\map{\quack{\spec[\dera{R}]}{\algr}}{\hoclo{\shi}}\end{equation} 
describes the space of $2$-forms on $\spec[\dera{R}]$ that are homotopically closed and $\algr$-invariant up to homotopy at the same time. Sometimes it is useful to have homotopically closed forms on $\spec[\dera{R}]$ that are strictly $\algr$-invariant in a sense to be made precise below. It becomes especially useful when every homotopically closed form on $\quack{\spec[\dera{R}]}{\algr}$ can be strictified in such a way. In this section we show this to be always true, if $\algr$ is linearly reductive. We denote by
	\begin{equation}\striclo{\shi}\lp\quack{\spec[\dera{R}]}{\algr}\rp\subseteq\staclo{\shi}\lp\quack{\spec[\dera{R}]}{\algr}\rp\end{equation}
the subcomplex consisting of $\underset{i,j\geq 0}\sum\sara^j\para^i\nform[j]{2+i}$ s.t.\@ each $\underset{i\geq 0}\sum\para^i\nform[0]{2+i}$ is invariant with respect to the action of $\algr$ on $\derham[\bullet]{\spec[\dera{R}]}$ induced by the action of $\algr$ on $\spec[\dera{R}]$, and $\forall j>0$, $\forall i\geq 0$ the form $\nform[j]{2+i}\in\derham[\bullet]{\spec[\dera{R}]\times\lp\algr^{\times^j}\rp}$ belongs to the ideal $\relide[j]\subseteq\derham[\bullet]{\spec[\dera{R}]\times\lp\algr^{\times^j}\rp}$ generated by $\pi_1^*\lp\derham[\bullet]{\algr^{\times^j}}\rp$, where $$\pi_1\colon\spec[\dera{R}]\times\lp\algr^{\times^j}\rp\longrightarrow\algr^{\times^j}$$ is the projection. We will call $\striclo{\shi}\lp\quack{\spec[\dera{R}]}{\algr}\rp$ \emph{the space of strictly $\algr$-invariant homotopically closed $2$-forms on $\quack{\spec[\dera{R}]}{\algr}$}.

\begin{proposition} Suppose that $\algr$ is linearly reductive, then 
		\begin{equation}\striclo{\shi}\lp\quack{\spec[\dera{R}]}{\algr}\rp\hooklongrightarrow\staclo{\shi}\lp\quack{\spec[\dera{R}]}{\algr}\rp\end{equation} 
is a weak equivalence of non-positively graded cochain complexes in $\vect$.\end{proposition}
\begin{proof} For each $j\geq 1$ the ideal $\relide[j]$ is clearly closed with respect to both $\homdi$ and $\dedi$. The subalgebra  $\pi_0^*\lp\derham[\bullet]{\spec[\dera{R}]}\rp\subseteq \derham[\bullet]{\spec[\dera{R}]\times\lp\algr^{\times^j}\rp}$, where $$\pi_0\colon\spec[\dera{R}]\times\lp\algr^{\times^j}\rp\rightarrow \spec[\dera{R}]$$ is the projection, is a complement to $\relide[j]$ as a graded submodule of $\derham[\bullet]{\spec[\dera{R}]\times\lp\algr^{\times^j}\rp}$.\footnote{Notice that $\pi_0^*\lp\derham[\bullet]{\spec[\dera{R}]}\rp$ is closed only with respect to $\homdi$.} Denoting $\relide[0]:=\lf0\rf\subseteq\derham[\bullet]{\spec[\dera{R}]}$ and applying $\negacy{\bullet}$ we have a cosimplicial diagram of inclusions
		\begin{equation}\forall j\geq 0\quad\negacy{\bullet}\lp\relide[j]\rp\subseteq\negacy{\bullet}\lp\derham[\bullet]{\spec[\dera{R}]\times\lp\algr^{\times^j}\rp}\rp.\end{equation}
Shifting upwards by $2-\shi$ and truncating to non-positive degrees we have another cosimplicial diagram
	\begin{equation}\forall j\geq 0\quad\negacy{\leq\shi-2}\lp\relide[j]\rp\subseteq\hoclo{\shi}\lp\spec[\dera{R}]\times\lp\algr^{\times^j}\rp\rp.\end{equation}
It is immediate to see that the cosimplicial diagram of quotients 
	\begin{equation}\lf\irre[j]\rf_{j\geq 0}:=\lf\hoclo{\shi}\lp\spec[\dera{R}]\times\lp\algr^{\times^j}\rp\rp/\negacy{\leq\shi-2}\lp\relide[j]\rp\rf_{j\geq 0}\end{equation} 
is exactly the diagram of Hochschild cochains for the action of $\algr$ on $\hoclo{\shi}\lp\spec[\dera{R}]\rp$ induced by the action on $\spec[\dera{R}]$. Since $\algr$ is reductive, taking the cosimplicial normalization of $\lf\irre[j]\rf_{j\geq 0}$, we obtain a non-negatively graded cochain complex of non-positively graded cochain complexes, s.t.\@ the vertical differential has $0$ cohomology everywhere, except in $\irre[0]$. Let $\kerre$ be the kernel of $\irre[0]\rightarrow\irre[1]$, we have an acyclic complex of non-positively graded complexes
	\begin{equation}\irre[0]/\kerre\longrightarrow\irre[1]\longrightarrow\irre[2]\longrightarrow\ldots,\end{equation}
hence its product-total complex is acyclic (e.g.\@ \cite{Wei} Lemma 2.7.3 p.\@ 59). \hide{See also The Stacks Project, Lemma 26.3  in Part 1 Chapter 12.} We notice that this product-total complex is exactly the quotient $$\staclo{\shi}\lp\quack{\spec[\dera{R}]}{\algr}\rp/\striclo{\shi}\lp\quack{\spec[\dera{R}]}{\algr}\rp.$$\end{proof}

Now we can have a more convenient reformulation of Proposition \ref{allgroups} in the reductive case.

\begin{proposition}\label{reductive} Let $\algr$ be a linearly reductive group acting on an affine dg manifold $\spec[\dera{R}]$. The normalization of the simplicial set of homotopically closed $2$-forms of degree $\shi$ on $\quack{\spec[\dera{R}]}{\algr}$ is weakly equivalent to
	\begin{equation}\striclo{\shi}\lp\quack{\spec[\dera{R}]}{\algr}\rp=\underset{r\leq 0}\bigoplus\;\striclo[r]{\shi}\lp\quack{\spec[\dera{R}]}{\algr}\rp,\end{equation}
	\begin{equation}\striclo[r]{\shi}\lp\quack{\spec[\dera{R}]}{\algr}\rp:=\lf\underset{i,j\geq 0}\sum\sara^j\para^i\nform[j]{2+i}\rf,\end{equation}
	where\begin{itemize}
		\item $\forall i\geq 0$ $\nform[0]{2+i}$ is a $\algr$-invariant $2+i$-form on $\spec[\dera{R}]$ of degree $r+d-2i$ and
		\item $\forall j\geq 1$, $\forall i\geq 0$ $\nform[j]{2+i}$ is a $2+i$-form on $\spec[\dera{R}]\times\lp\algr^{\times^j}\rp$ of degree $r+\shi-2-2i-2j$ belonging to the ideal generated by $\derham[\bullet]{\algr^{\times^j}}$.\end{itemize} 
	The differential on $\striclo{\shi}\lp\quack{\spec[\dera{R}]}{\algr}\rp$ is $\homdi+\para\dedi+\sara\stadi$, where $\stadi$ is the alternating sum of all pullbacks over face maps in $\lf\spec[\dera{R}]\times\lp\algr^{\times^j}\rp\rf_{j\in\noneg}$. Elements of $\striclo[0]{\shi}\lp\quack{\spec[\dera{R}]}{\algr}\rp$ have to be $\homdi+\para\dedi+\sara\stadi$-cocycles.
	
	Any homotopically closed $2$-form of degree $\shi$ on $\quack{\spec[\dera{R}]}{\algr}$ can be described as an element of  $\striclo[0]{\shi}\lp\quack{\spec[\dera{R}]}{\algr}\rp$. Two such forms are equivalent if they differ by a $\homdi+\para\dedi+\sara\stadi$-coboundary.\end{proposition}
Remark \ref{hypercohomology} also gets a reductive reformulation, stating that gluing strictly $\algr$-invariant homotopically closed forms on intersections of affine charts can be done using strictly $\algr$-invariant forms again. Explicitly we have the following

\begin{proposition} Let $\algr$ be a linearly reductive group acting on a dg manifold $\dgsch$, s.t.\@ the action on $\dgsch[0]$ admits a good quotient. Let $\quospa$ be the topological space underlying $\quo{\dgsch}{\algr}$. Then $\striclo{\shi}$ defines a sheaf of non-positively graded $\mathbb C$-linear cochain complexes on $\quospa$ and
	\begin{equation}\hyperco[\bullet]\lp\quospa,\striclo{\shi}\rp\cong\hyperco[\bullet]\lp\quospa,\staclo{\shi}\rp.\end{equation}
In particular the space of homotopically closed $2$-forms of degree $\shi$ on $\quack{\dgsch}{\algr}$ is isomorphic to $\hyperco[0]\lp\quospa,\striclo{\shi}\rp$.\end{proposition}
\begin{proof} According to Prop.\@ \ref{reductive} the inclusion $\striclo{\shi}\hookrightarrow\staclo{\shi}$ is a local weak equivalence of sheaves of cochain complexes. Therefore the induced map on hypercohomology groups is an isomorphism.\end{proof}

\medskip

Now we turn to integrable distributions. This is an old concept, and accordingly there are many names for it in different kinds of geometry: Lie--Rinehart algebras, Lie algebroids, foliations etc. Keeping the notation compatible with \cite{BSY} we use the term \emph{integrable distributions}. We will consider two equivalent ways to define these. First: an integrable distribution on an affine dg manifold $\spec[\dera{R}]$ is given by a perfect dg $\dera{R}$-module $\lmin$ with generators in non-negative degrees, a $\mathbb C$-linear dg Lie algebra structure on $\lmin$ and an $\dera{R}$-linear \emph{anchor map} $\anchor\colon\lmin\rightarrow\tanga{\spec[\dera{R}]}$ that is a morphism of dg Lie algebras and satisfies the well known conditions (e.g.\@ \cite{Mackenzie} Def.\@ 3.3.1 p.\@ 100). 

The second way is an $\dera{R}$-linear dual formulation using Koszul duality: an integrable distribution is defined as a morphism of graded mixed algebras
	\begin{equation}\label{coanchor}\coanchor\colon\derham[\bullet]{\dera{R}}\longrightarrow\gramial\end{equation}
satisfying some conditions (\cite{P14}, \cite{AlgFoliations} Def.\@ 1.2, or the strictified version \cite{BSY} Def.\@ 3). These conditions force us to consider a proper subcategory of the category of all possible morphisms of graded mixed algebras (\ref{coanchor}), and then there is a categorical equivalence between the two definitions.

\smallskip

To define integrable distributions on stacks one needs contravariant functoriality, i.e.\@ to be able to pull back an integrable distribution over a morphism of affine dg manifolds. This is also an old construction (\cite{HM} p.\@ 203), and in using it one obtains a stack of integrable distributions on $\saff$ (\cite{P14}, \cite{AlgFoliations} Prop.\@ 1.2.3). Here we need to be careful what kind of stacks we are talking about. 

For each $\spec[\dera{R}]$ there is a model category $\alldis{\spec[\dera{R}]}$ of all integrable distributions on $\spec[\dera{R}]$, and we can view it as an $\infty$-category (e.g.\@ a category enriched in $\sset$). The stack considered in \cite{AlgFoliations} is a stack of $\sset$-categories. Given a stack $\astack$ on $\saff$ the $\infty$-category of integrable distributions on $\astack$ is defined then (\cite{AlgFoliations} Def.\@ 1.2.4) as
	\begin{equation}\alldis{\astack}:=\underset{\spec[\dera{R}]\rightarrow\astack}\holim\alldis{\spec[\dera{R}]},\end{equation}
computed in the category of $\sset$-categories. In this paper we are not interested in all of the resulting $\infty$-category, but only in the maximal $\infty$-subgroupoid.\footnote{I.e.\@ the largest $\sset$-subcategory, whose category of connected components is a groupoid.} A $\sset$-enriched groupoid can be equivalently described by its nerve, and it is this simplicial set that we would like to investigate (in fact just the set of connected components in it).

There is an equivalent way to obtain the same simplicial set. The homotopically coherent nerve construction gives us a right Quillen functor in a Quillen equivalence between model categories of $\sset$-categories and quasi-categories (\cite{Lurie} Thm.\@ 2.2.5.1 p.\@ 89). In turn there is a right Quillen functor from the category of quasi-categories to the category of simplicial sets with the usual model structure (\cite{JoyalT} Thm.\@ 1.19 p.\@ 283) which extracts the largest Kan subcomplex out of a quasi-category (\cite{JoyalT} Prop.\@ 1.16 p.\@ 283, Prop.\@ 1.20 p.\@ 284). Altogether this amounts to extracting the largest simplicial sub-groupoid from a $\sset$-category and taking its homotopically coherent nerve, or equivalently the usual nerve (\cite{CohNerve} \S2.6). For each $\spec[\dera{R}]$ we denote by $\kan[\alldis{\spec[\dera{R}]}]$ the resulting simplicial set. Since this is a right Quillen functor we have (\cite{Hirsh} Thm.\@ 19.4.5 p.\@ 415)
	\begin{equation}\kan[\underset{\spec[\dera{R}]\rightarrow\astack}\holim\alldis{\spec[\dera{R}]}]\simeq\underset{\spec[\dera{R}]\rightarrow\astack}\holim\kan[\alldis{\spec[\dera{R}]}].\end{equation}
This means that we can restrict to the maximal $\infty$-subgroupoids $\alldig{\spec[\dera{R}]}\subseteq\alldis{\spec[\dera{R}]}$ from the beginning and we define \emph{the stack of integrable distributions} to be 
		\begin{equation}\stadis\colon\spec[\dera{R}]\longmapsto\nerve{\alldig{\spec[\dera{R}]}},\end{equation}
where $\nerve{-}$ stands for the usual nerve construction. Then an integrable distribution on a stack $\astack$ is given by a map $\astack\rightarrow\stadis$.

\smallskip

Let us look at $\map{\quack{\spec[\dera{R}]}{\algr}}{\stadis}$. As before, using just cofibrancy in each simplicial dimension, we find  that $$\map{\quack{\spec[\dera{R}]}{\algr}}{\stadis}\simeq\holim[j\geq 0]\stadis\lp\spec[\dera{R}]\times\lp\algr^{\times^j}\rp\rp.$$ We would like to go further and claim that the latter simplicial set is weakly equivalent to 
	\begin{equation}\label{points}\map{\stasi}{\lf\stadis\lp\spec[\dera{R}]\times\lp\algr^{\times^j}\rp\rp\rf_{j\geq 0}}.\hide{\footnote{Assuming that $\spec[\dera{R}]$ is dg affine manifold implies that the graded mixed algebra $\derham[\bullet]{\dera{R}}$ is in the correct weak equivalence class, and hence the resulting $\alldig{\spec[\dera{R}]}$ has the correct homotopy type.}}\end{equation} 
For this claim to be true, it is enough to require that $\lf\stadis\lp\spec[\dera{R}]\times\lp\algr^{\times^j}\rp\rp\rf_{j\geq 0}$ is fibrant in the Reedy model structure on the category of cosimplicial diagrams of simplicial sets. It is easier to check this property, if we work with cosimplicial diagrams of $\infty$-groupoids instead. 

The $j$-th matching object is the $\infty$-groupoid of integrable distributions on the union of images of degeneracies in $\spec[\dera{R}]\times\lp\algr^{\times^j}\rp$. Then the Reedy fibrancy condition amounts to: $\forall j\geq 0$ the functor from $\alldig{\spec[\dera{R}]\times\lp\algr^{\times^j}\rp}$ to the $j$-th matching groupoid is a fibration of $\sset$-categories, i.e.\@ it consists of fibrations between mapping spaces and it lifts weak equivalences (\cite{Bergner} p.\@ 2044). Choosing fibrant replacements of $\spec[\dera{R}]$ and $\algr$, so that all algebras of functions are almost free, and requiring all integrable distributions to be represented by almost free morphisms of graded mixed algebras, it is straightforward to see that this fibrancy condition is satisfied.

Using (\ref{points}) to provide an explicit description of integrable distributions on $\quack{\spec[\dera{R}]}{\algr}$ we obtain that a $0$-simplex in this simplicial set can be described as an integrable distribution on $\spec[\dera{R}]$ and a coherent system of weak equivalences between all possible pullbacks of this distribution to $\lf\spec[\dera{R}]\times\lp\algr^{\times^j}\rp\rf_{j\geq 1}$ (a similar definition in the non-derived context is given in \cite{Waldron} \S5.2). 

\smallskip

Given a not necessarily affine dg manifold $\dgsch$ with an action of $\algr$, s.t.\@ $\dgsch$ has an atlas consisting of $\algr$-invariant dg affine manifolds, and $\dgsch[0]$ admits a good quotient, we would like to have a description of the simplicial set $\map{\quack{\dgsch}{\algr}}{\stadis}$. For each $\algr$-invariant chart $\spec[\dera{R}]$ on $\dgsch$  we have the simplicial set of integrable distributions on $\quack{\spec[\dera{R}]}{\algr}$, and as in \cite{BSY} Def.\@ 7 we can consider equivalence classes of integrable distributions on arbitrary parts of $\algr$-invariant affine atlases on $\dgsch$ (this construction involves the entire space (\ref{points}), not just its $0$-simplices). This gives us a sheaf $\dishe[\algr]{\dgsch}$ on the space $\quospa$ of classical points in $\quack{\dgsch}{\algr}$, sections of which correspond to equivalence classes of integrable distributions, and we have
	\begin{equation}\Gamma\lp\quospa,\dishe[\algr]{\dgsch}\rp\cong\pi_0\lp\map{\quack{\dgsch}{\algr}}{\stadis}\rp.\end{equation}
As in \cite{BSY} we would like to restrict our attention to integrable distributions that do not have non-trivial ``isotropy dg Lie subalgebras''. Precisely we call an integrable distribution on $\quack{\spec[\dera{R}]}{\algr}$ \emph{a derived foliation}, if around every $\mathbb C$-point of $\quack{\spec[\dera{R}]}{\algr}$ there is a minimal\footnote{An affine dg manifold is \emph{minimal} at a point, if the complex of K\"ahler differentials has $0$ differential at this point.}  $\algr$-invariant chart $\spec[\dera{R_1}]$, where the distribution can be written as a quotient $\coanchor_1\colon\derham[\bullet]{\dera{R_1}}\twoheadrightarrow\gramial$ (\cite{BSY} Def.\@ 6), or in other words the anchor map is injective. Notice that we require injectivity of the anchor in all degrees, not just in degree $0$ as for rigid distributions in \cite{AlgFoliations} Def.\@ 1.2.6. See also the complementary notion of a co-foliation in \cite{cofoliations} \S3.

\hide{\begin{remark} We observe that the notion of a derived foliation is independent of a presentation of the integrable distribution, i.e.\@ it is a condition on the set of weak equivalence classes of distributions, defining a full $\infty$-subgroupoid $\fol{\spec[\dera{R}]}\subseteq\stadis\lp\spec[\dera{R}]\rp$ (or equivalently the union of a set of connected components in the simplicial set). Since $\stadis\lp\quack{\spec[\dera{R}]}{\algr}\rp$ can be described as a full $\infty$-subgroupoid of $\stadis\lp\spec[\dera{R}]\rp$ equipped with extra structure, taking the pullback we obtain a full $\infty$-subgroupoid $\fol{\quack{\spec[\dera{R}]}{\algr}}\subseteq\stadis\lp\quack{\spec[\dera{R}]}{\algr}\rp$.\end{remark}}

Having injective anchor maps allows us to consider a strictification of the notion of $\algr$-invariance of integrable distributions. We will say that a derived foliation on $\quack{\spec[\dera{R}]}{\algr}$ is \emph{a strictly $\algr$-invariant derived foliation}, if:\begin{enumerate}
	\item this integrable distribution can be written on $\spec[\dera{R}]$ itself using an injective anchor,
	\item the two pullbacks to $\spec[\dera{R}]\times\algr$ are canonically isomorphic, i.e.\@ the corresponding ideals in $\derham[\bullet]{\spec[\dera{R}]\times\algr}$ are equal, and
	\item the coherent system of weak equivalences between the various (equal) pullbacks of this integrable distribution to $\lf\spec[\dera{R}]\times\lp\algr^{\times^j}\rp\rf_{j\geq 1}$ are the identities.\end{enumerate}
Let $\fol[\algr]{\dgsch}\subseteq\dishe[\algr]{\dgsch}$ be the subsheaf consisting of equivalence classes of integrable distributions that Zariski locally on $\dgsch$ can be written as strictly $\algr$-invariant derived foliations. If $\algr$ is trivial, i.e.\@ we have just a dg manifold $\dgsch$, we will write $\fol{\dgsch}$ for the resulting sheaf.

\medskip

Now we turn to isotropic structures on integrable distributions. An integrable distribution on $\spec[\dera{R}]$ is defined as a morphism of graded mixed algebras $\coanchor\colon\derham[\bullet]{\spec[\dera{R}]}\rightarrow\gramial$, thus each one gives us a cochain complex $\negacy{\shi}\lp\gramial\rp$ and, since weak equivalences between integrable distributions produce weak equivalences between the corresponding negative cyclic complexes we have a fibration of simplicial sets $\negdi\lp\spec[\dera{R}]\rp\rightarrow\stadis\lp\spec[\dera{R}]\rp$ with fibers being the negative cyclic complexes.

Applying $\negacy{\shi}$ to $\coanchor$ we obtain a morphism $\hoclo{\shi}\lp\spec[\dera{R}]\rp\times\stadis\lp\spec[\dera{R}]\rp\rightarrow\negdi\lp\spec[\dera{R}]\rp$ in the category of simplicial sets over $\stadis\lp\spec[\dera{R}]\rp$, and it is clearly functorial in $\spec[\dera{R}]$, i.e.\@ we have a natural transformation $\hoclo{\shi}\times\stadis\rightarrow\negdi$ over $\stadis$. Evaluating this at the simplicial diagram $\lf\spec[\dera{R}]\times\lp\algr^{\times^j}\rp\rf_{j\geq 0}$ we get the corresponding cosimplicial diagram of maps of simplicial sets. Now suppose we have a choice $\derham[\bullet]{\spec[\dera{R}]}\rightarrow\gramial$ of a strictly $\algr$-invariant derived foliation on $\quack{\spec[\dera{R}]}{\algr}$, this gives us a map from the \emph{constant} cosimplicial-simplicial set on one point to $\lf\stadis\lp\spec[\dera{R}]\times\lp\algr^{\times^j}\rp\rp\rf_{j\geq 0}$. Taking the fiber over this point of  $\hoclo{\shi}\times\stadis\rightarrow\negdi$ we obtain a morphism of cosimplicial-simplicial objects in $\vect$:
	\begin{equation}\label{manyan}\forall j\geq 0\quad\hoclo{\shi}\lp\spec[\dera{R}]\times\lp\algr^{\times^j}\rp\rp\longrightarrow\negacy{\shi}\lp\pi_j^*\lp\gramial\rp\rp,\end{equation}
where $\pi_j\colon\spec[\dera{R}]\times\lp\algr^{\times^j}\rp$ is the projection. A homotopically closed $2$-form of degree $\shi$ on $\quack{\spec[\dera{R}]}{\algr}$ is given by a map $\stasi\rightarrow\lf\hoclo{\shi}\lp\spec[\dera{R}]\times\lp\algr^{\times^j}\rp\rp\rf_{j\geq 0}$, and composing with (\ref{manyan}) we obtain $\eva\colon\stasi\rightarrow\lf\negacy{\shi}\lp\pi_j^*\lp\gramial\rp\rp\rf_{j\geq 0}$. On the other hand there is the canonical $0$-map. Then \emph{an isotropic structure on $\gramial$} is defined as a homotopy 
	\begin{equation}\label{homotopy}\stasi[1]\times\stasi\longrightarrow\lf\negacy{\shi}\lp\pi_j^*\lp\gramial\rp\rp\rf_{j\geq 0}\end{equation}
between $\eva$ and the $0$-map.\footnote{Recall that $\stasi=\lf\stasi[j]\rf_{j\geq 0}$ is the cosimplicial diagram of the standard simplices in $\sset$.}With obvious modifications this construction can be used for integrable distributions on $\quack{\spec[\dera{R}]}{\algr}$ that are not strictly $\algr$-invariant derived foliations. We have considered only the special case because then we can write a (\ref{homotopy}) explicitly.

Recall that we can use cosimplicial-simplicial normalization to transfer computation of mapping spaces from cosimplicial-simplicial objects in $\vect$ to graded mixed complexes. This transfer turns $\stasic$ into a cofibrant replacement of $\mathbb C$ considered to be in degree 0 and weight $0$. It is easy to see that normalization of $\stasic[1]\times\stasic$ is weakly equivalent to the graded mixed complex that is a cofibrant replacement of the complex $\mathbb C\hookrightarrow\mathbb C\oplus\mathbb C$ placed in weight $0$ and degrees $-1$ and $0$. This gives us the following description of isotropic structures on strictly $\algr$-invariant derived foliations.

Let $\symstra=\underset{i,j\geq 0}\sum\sara^j\para^i\nform[j]{2+i}$ be a homotopically closed $2$-form of degree $\shi$ on $\quack{\spec[\dera{R}]}{\algr}$. A strictly $\algr$-invariant derived foliation $\coanchor\colon\derham[\bullet]{\spec[\dera{R}]}\rightarrow\gramial$ is \emph{isotropic with respect to $\symstra$} if $\coanchor\lp\underset{i\geq 0}\sum\para^i\formstra[0]{i}\rp$ is a $\homdi+\para\dedi$-coboundary. \emph{An isotropic structure} on a strictly $\algr$-invariant isotropic derived foliation is given by $\lagra:=\underset{i,j\geq 0}\sum\sara^j\para^i\lagrast[j]{i}\in\negacy{\bullet}\lp\gramial\rp$ s.t.\@
	\begin{equation}\lp\homdi+\para\dedi+\sara\stadi\rp\lp\underset{i\geq 0}\sum\para^i\lagrast[j]{i}\rp=\coanchor\lp\underset{i\geq 0}\sum\para^i\nform[0]{i}\rp.\end{equation} 
Notice that we require (up to homotopy) vanishing of $\symstra$ only on the derived foliation on $\spec[\dera{R}]$ and say nothing about $\underset{i\geq 0}\sum\sara^j\para^i\formstra[j]{2+i}$ for $j>0$. This is a consequence of the normalization of $\stasic[1]\times\stasic$ being so simple. We would also like to define \emph{strictly $\algr$-invariant isotropic structures} on strictly $\algr$-invariant derived foliations as those isotropic structures that consist of $\underset{i\geq 0}\sum\para^i\lagrast{i}$ only, in which case the two pullbacks of $\underset{i\geq 0}\sum\para^i\lagrast{i}$ to $\spec[\dera{R}]\times\algr$ have to be equal.\footnote{Having required that the derived foliation is strictly $\algr$-invariant we can directly compare the two pullbacks of isotropic structures.} 

Being an isotropic distribution is clearly a local condition so given a homotopically closed $2$-form $\symstra$ we define a subsheaf $\dish[\algr]{\dgsch}{\symstra}\subseteq\fol[\algr]{\dgsch}$ consisting of sections whose corresponding foliations are isotropic with respect to $\symstra$. Choosing isotropic structures on isotropic foliations gives us another sheaf $\sish[\algr]{\dgsch}{\symstra}$ with a forgetful map $\sish[\algr]{\dgsch}{\symstra}\rightarrow\dish[\algr]{\dgsch}{\symstra}$. 

Now suppose that $\symstra$ is symplectic. For an isotropic structure to be Lagrangian with respect to $\symstra$ is a local condition, so we can restrict to $\quack{\spec[\dera{R}]}{\algr}$. There we have the tangent complex $\tanga{\quack{\spec[\dera{R}]}{\algr}}$ that is a dg $\dera{R}$-module concentrated in degrees $\geq -1$. For a strictly $\algr$-invariant derived foliation $\lmin\hookrightarrow\tanga{\spec[\dera{R}]}$ an isotropic structure $\lagra$ is \emph{Lagrangian}, if it defines a (shifted) weak equivalence between the homotopy kernel of the composite morphism 
\begin{equation}\lmin\hookrightarrow\tanga{\spec[\dera{R}]}\rightarrow\tanga{\quack{\spec[\dera{R}]}{\algr}}\end{equation}
and the dual of $\lmin$. For details see e.g.\@ \cite{BSY} Def.\@ 12. We denote by $\lish[\algr]{\dgsch}{\symstra}\subseteq\sish[\algr]{\dgsch}{\symstra}$ the subsheaf consisting of isotropic distributions and isotropic structures that are Lagrangian.

\medskip

We finish this section with some statements that follow from definitions and Prop.\@ \ref{reductive}.

\begin{proposition}\label{trivial} Suppose $\spec[\dera{R}]\cong\spec[\dera{B}]\times\lp\subg\setminus\algr\rp$ for a closed subgroup $\subg$ and the action of $\algr$ on $\spec[\dera{R}]$ is through $\subg\setminus\algr$. Then the pullback functor over $\spec[\dera{R}]\twoheadrightarrow\spec[\dera{B}]$ identifies the following sheaves on the space $\quobit$ of classical points in $\quack{\spec[\dera{R}]}{\algr}$: the sheaf $\fol[\algr]{\spec[\dera{R}]}$ of equivalence classes of strictly $\algr$-invariant derived foliations on $\spec[\dera{R}]$ and the sheaf $\fol{\spec[\dera{B}]}$ of equivalence classes of derived foliations on $\spec[\dera{B}]$.

Suppose that $\algr$ is linearly reductive, then any homotopically closed $2$-form on $\quack{\spec[\dera{R}]}{\algr}$ can be written as a strictly $\algr$-invariant $\underset{i,j\geq 0}\sum\sara^j\para^i\formstra[j]{i}$, where $\underset{i\geq 0}\sum\para^i\formstra[0]{i}$ is pulled back from $\spec[\dera{B}]$. This pullback functor also identifies the following sheaves on the space $\quobit$:\begin{enumerate}[label=(\roman*)]
	\item the sheaf $\dish[\algr]{\spec[\dera{R}]}{\symstra}$ of equivalence classes of strictly $\algr$-invariant isotropic derived foliations on $\spec[\dera{R}]$ and the sheaf $\dish{\spec[\dera{B}]}{\symstra}$ of equivalence classes of isotropic derived foliations on $\spec[\dera{B}]$,
	\item the sheaf $\sish[\algr]{\spec[\dera{R}]}{\symstra}$ of equivalence classes of strictly $\algr$-invariant isotropic derived foliations on $\spec[\dera{R}]$ equipped with strictly $\algr$-invariant isotropic structures and the sheaf $\sish{\spec[\dera{B}]}{\symstra}$ of equivalence classes of isotropic derived foliations on $\spec[\dera{B}]$ equipped with isotropic structures.
\end{enumerate}\end{proposition}
\hide{\begin{proof} Since the projection $\spec[\dera{R}]\twoheadrightarrow\spec[\dera{B}]$ has a section\hide{ (e.g.\@ given by the class of $\subg$)} clearly $\fol{\spec[\dera{B}]}$ is a subsheaf of $\fol[\algr]{\spec[\dera{R}]}$. To prove that this inclusion is an isomorphism of sheaves we work locally and assume that $\spec[\dera{B}]$ is minimal at some $\pt\in\quobit$. Then a section of $\fol[\algr]{\spec[\dera{R}]}$ around $\pt$ is represented by a surjective $\coanchor\colon\derham[\bullet]{\dera{R}}\twoheadrightarrow\gramial$ whose two pullbacks to $\derham[\bullet]{\spec[\dera{R}]\times\algr}$ generate the same ideal. We have $\derham{\dera{R}}\cong\derham{\dera{B}}\boxplus\derham{\subg\setminus\algr}$ and the two pullbacks of $\derham{\subg\setminus\algr}$ to $\derham{\spec[\dera{R}]\times\algr}$ have trivial intersection. Therefore
	\begin{equation}\Ker{\coanchor}\cap\derham{\subg\setminus\algr}=\lf 0\rf,\end{equation}
	in other words any strictly $\algr$-invariant derived foliation has to contain the tangent vectors to the orbits of $\algr$. Thus we see that $\lmin\hookrightarrow\tanga{\spec[\dera{R}]}$ given by $\coanchor$ is a direct sum of $\tanga{\subg\setminus\algr}$ and a derived foliation on $\spec[\dera{B}]$. So $\fol{\spec[\dera{B}]}\hookrightarrow\fol[\algr]{\spec[\dera{R}]}$ is surjective.\end{proof}}

\section{Lagrangian distributions in the $-2$-shifted case}\label{three}

In this section we come back to the particular case of quotient stacks of $\quot$-schemes. First we assume that $\CY$ is a Calabi--Yau manifold of dimension $2-\shi$ for some $\shi\in\mathbb Z_{<0}$. According to \cite{PTVV13} \S2.1 the moduli stack of perfect complexes on $\CY$, as defined in \cite{ToVa}, carries a $\shi$-shifted symplectic structure. As it is shown in \cite{BKS} there is a formally \'etale morphism to this stack from $\quack{\derqof{\twi{\sumto{\twish{\CY}}{\dimglo}}{-\casmu}}{\hilpo}{\lebo,\ribo}}{\GL{\dimglo}}$, thus we can pull back the symplectic structure to the latter stack. Restricting to the stable part we obtain a $\shi$-shifted symplectic structure $\symstra$ on $\quack{\staqof{\twi{\sumto{\twish{\CY}}{\dimglo}}{-\casmu}}{\hilpo}{\lebo,\ribo}}{\GL{\dimglo}}$. In this section the data $\CY$, $\hilpo$, $\lebo$, $\ribo$, $\casmu$ and $\dimglo$ will remain constant, so for typographical reasons we suppress them from the notation and write simply $\quack{\derquot}{\GL{\dimglo}}$.

Our goal in this paper is to show existence of a special kind of globally defined Lagrangian distributions on $\quack{\derquot}{\GL{\dimglo}}$. For that we restrict our attention to Calabi--Yau manifolds of dimension $4$, i.e.\@ $\shi=-2$. In addition we will switch from the algebraic geometry over $\mathbb C$ to $\cinfty$-geometry over $\mathbb R$. All of our dg manifolds become then derived $\cinfty$-manifolds, and we will denote this change by underlining. In particular we will write $\demquot$ for the derived $\cinfty$-manifold underlying $\derquot$. We will also use $\smoup$ to denote the Lie group of invertible $\dimglo\times\dimglo$-matrices with coefficients in $\mathbb C$, similarly for $\spoup$.

Switching to $\cinfty$-geometry changes the topology. We will use the usual metric topology on all of our $\cinfty$-manifolds. Since it is stronger than the \'etale topology, Proposition \ref{LocalProduct} translates into a proposition stating that $\demquot$ is a principal $\spoup$-bundle over the quotient
\begin{equation}\quo{\demquot}{\spoup}.\footnote{As before the symbol $\sslash$ means taking functions that are invariant with respect to the group action.}\end{equation}

Switching from $\mathbb C$ to $\mathbb R$ breaks our shifted symplectic structures into the real and imaginary parts: $\refo{\symstra}$, $\imfo{\symstra}$. As in \cite{BSY} \S4 both parts will play a role. We will seek Lagrangian distributions with respect to $\imfo{\symstra}$, that are also negative definite with respect to $\refo{\symstra}$. To explain the second condition we need to put some additional cohomological restrictions on our integrable distributions. 

As in \cite{BSY} Def.\@ 15 on p.\@ 23 we will say that an integrable distribution $\anchor\colon\lmin\rightarrow\tanga{\derma}$ on a derived $\cinfty$-manifold $\derma$ is \emph{a purely derived foliation} if $H^{\leq 0}\lp\lmin,\homdi\rp=0$ and on a minimal chart around every classical point in $\derma$ a representative of $\lp\lmin,\anchor\rp$ can be chosen s.t.\@ $\anchor$ is an inclusion of complexes (see \cite{BSY} Rem.\@ 5 p.\@ 14). If there is a Lie group acting on $\derma$ we will consider strictly invariant purely derived foliations on $\derma$.

Now a Lagrangian distribution with respect to $\imfo{\symstra}$ is \emph{negative definite} with respect to $\refo{\symstra}$, if $\refo{\symfo}$ defines a negative definite $2$-form on $H^1\lp\lmin,\homdi\rp$ (\cite{BSY} Def.\@ 20 p.\@ 30). This condition is clearly invariant with respect to equivalences of integrable distributions and it is local on the space $\quospa$ of classical points in $\quack{\demquot}{\smoup}$. Hence it defines a subsheaf of the sheaf of all Lagrangian distributions with respect to $\imfo{\symstra}$, that we denote by
	\begin{equation}\nish[\smoup]{\demquot}{{\symstra}}\hooklongrightarrow\lish[\smoup]{\demquot}{\imfo{\symstra}}.\end{equation}
We would like to show that this subsheaf has at least one global section. This immediately follows from the following
\begin{proposition}\label{existence} The sheaf $\nish[\smoup]{\demquot}{{\symstra}}$ on $\quospa$ is soft and the set of germs around every point in $\quospa$ is not empty.\end{proposition} 	
\begin{proof} Since the topological space $\quospa$ is Hausdorff and second countable the property for a sheaf on $\quospa$ to be soft is local, i.e.\@ it is enough to show it in an arbitrarily small neighbourhood of every point. Let $\derma$ be a $\smoup$-invariant chart on $\demquot$ that is small enough so that we have $\derma\cong\derna\times\spoup$, and the action of $\smoup$ is through the action on $\spoup$. Let $\quobit$ be the corresponding open subset of $\quospa$. Using a $\cinfty$-reformulation of Prop.\@ \ref{trivial} we see that restriction of $\nish[\smoup]{\demquot}{\symstra}$ to $\quobit$ is isomorphic to a subsheaf $\nish{\derna}{\symstra}\subseteq\dish{\derna}{\imfo{\symstra}}$, whose sections correspond to equivalence classes of purely derived foliations with isotropic structures, having the additional properties that they are negative definite with respect to $\refo{\symstra}$ and Lagrangian when seen as strictly invariant distributions on $\quack{\derna}{\smgl{1}}$ (the action being trivial).
	
Applying Thm.\@ 2.10 in \cite{BBBJ} p.\@ 1302 before the switch to $\cinfty$-geometry, we see that we can construct a morphism $\derna\rightarrow\deroa$ of derived $\cinfty$-manifolds, s.t.\@ $\imfo{\symstra}$, $\refo{\symstra}$ are pullbacks of the imaginary and real parts of a $-2$-shifted \emph{symplectic} structure on $\deroa$.\footnote{Notice that because of non-triviality of $H^3\lp\tanga{\derna},\homdi\rp$ the derived $\cinfty$-manifold $\derna$ cannot carry a $-2$-shifted symplectic structure. Only the stack $\quack{\derna}{\smgl{1}}$ does.} The map $\derna\rightarrow\deroa$ identifies the spaces of classical points and the corresponding pullback functor gives a morphism of sheaves 
	\begin{equation}\label{BBBJ}\dish{\deroa}{\imfo{\symstra}}\longrightarrow\dish{\derna}{\imfo{\symstra}}.\end{equation} 
We claim that this is an inclusion of a subsheaf. It is enough to prove it locally, so let $\pt\in\quobit$ and we can assume that $\derna$ is minimal at $\pt$. This implies that $\imfo{\symstra}$ defines a shifted isomorphism $\tanga{\quack{\derna}{\smgl{1}}}\rightarrow\derham{\quack{\derna}{\smgl{1}}}$ in a neighbourhood of $\pt$ in $\derna[0]$. Since the action of $\smgl{1}$ is trivial, the part $\tanga[0]{\smgl{1}}\subseteq\tanga[-1]{\quack{\derna}{\smgl{1}}}$ consists of $\homdi$-cocycles. Moreover, using \cite{BBBJ} Prop.\@ 2.4 p.\@ 1296 we can choose a representative of the $-2$-shifted symplectic form on $\quack{\derna}{\smgl{1}}$ s.t.\@ the component on $\derna\times\smgl{1}$ is a closed $2$-form (of degree $-3$). This implies that the part of $\derham{\derna}$ of degree $-3$ that corresponds to $\tanga[0]{\smgl{1}}$ under the isomorphism given by $\imfo{\symstra}$ consists of $\dedi$-cocycles. Locally these are also $\dedi$-coboundaries, and hence around $\pt$  in $\derna$ we can choose $2$ functions of degree $-3$, that are $\homdi$-cocycles and whose de Rham differentials are linearly independent at $\pt$. Dividing by the ideal these functions generate we obtain a section of $\derna\twoheadrightarrow\deroa$. Thus (\ref{BBBJ}) is injective.

It is clear that (\ref{BBBJ}) cannot be surjective, since the two vector fields in degree $3$ corresponding to the two functions above generate a derived foliation that becomes $0$ when pulled back to $\deroa$. In fact locally around $\pt$ we can write $\derna\cong\deroa\times\lp\suspense[-3]{\gl{1}}\rp$, where $\suspense[-3]{\gl{1}}$ is the derived $\cinfty$-manifold whose dg $\cinfty$-ring is freely generated by two functions of degree $-3$. Obviously every integrable distribution that is pulled back over $\derna\twoheadrightarrow\deroa$ has to contain the tangent vectors along $\suspense[-3]{\gl{1}}$, and conversely every integrable distribution containing this bundle is in the pullback. Now we notice that because of the purely derived condition distributions corresponding to sections of $\nish{\derna}{\symstra}$ have this property, i.e.\@ $\nish{\derna}{\symstra}\subseteq\dish{\deroa}{\imfo{\symstra}}$. We can recognize this subsheaf of $\dish{\deroa}{\imfo{\symstra}}$ as the sheaf $\nish{\deroa}{\symstra}$, whose sections correspond to purely derived foliations on $\deroa$, that are Lagrangian with respect to $\imfo{\symstra}$ and negative definite with respect to $\refo{\symstra}$. In \cite{BSY} Thm.\@ 3 this sheaf is shown to be soft and possess local sections.\end{proof}

\medskip

Proposition \ref{existence} shows that there are globally defined purely derived foliations on the stack $\quack{\demquot}{\smoup}$ that are Lagrangian with respect to $\imfo{\symstra}$ and negative definite with respect to $\refo{\symstra}$. In general such an integrable distribution consists of distributions on $\smoup$-invariant pieces with coherent gluing data on various intersections. In this particular case our derived stack is especially nice, it is given as a quotient stack of a derived $\cinfty$-manifold by an action of a Lie group. One would expect that in such situations there should be one globally defined derived foliation on the entire derived $\cinfty$-manifold without any need for gluing.

This is indeed what happens in this case. First of all, since we are working in derived $\cinfty$-geometry all of our derived manifolds are locally fibrant, i.e.\@ locally their dg $\cinfty$-rings of functions are almost free. This implies that we can find local representatives of any derived foliation on $\demquot$ itself. Let $\derma$ be a chart on $\demquot$ where we have such a representative $\lmin$. By construction it is a purely derived foliation and we can use Prop.\@ 8 from \cite{BSY}. So we can assume that $\lmin$ is a subcomplex $\tanga{\derma}$, that has generators only  in degrees $\geq 1$ and consist of all of $\tanga{\derma}$ in degrees $\geq 2$. As in the proof of Prop.\@ 8 in \cite{BSY} we can choose a sequence $\lf E_k\rf_{1\leq k\leq r}$ of vector bundles on $\derma[0]$ that freely generate $\lmin$ as a graded module, with $E_k$ sitting in degree $k$. Moreover, we can assume that each $E_k$ is a subbundle of $F_k$, where $\lf F_k\rf_{0\leq k\leq r}$ are some generating bundles for $\tanga{\derma}$ (clearly $E_k\cong F_k$ for $k\geq 2$).

We would like to glue the patchwork of subcomplexes into one globally defined subcomplex of $\tanga{\demquot}$. This requires finding a canonical complement of $\lmin$ in $\tanga{\derma}$. First of all we note that because the Euler characteristic of $\lmin$ is half of that of $\tanga{\derma}$ the rank of $E_1$ is independent of the representative $\lmin$ we have chosen. Moreover, denoting by $Z^1$ the sub-module of $\homdi$-cocycles in $F_1$ we have a canonical complement to $E_1$ in $F_1$:
	\begin{equation}E_1':=\lp E_1\cap Z^1\rp^{\perp_{\refo{\symstra}}}\cap Z^1,\end{equation}
which is a subbundle. Therefore, given two charts $\derma$, $\derna$ and two representatives of the foliation $\lmin$, $\llin$, we have a map $\llin\twoheadrightarrow\lmin$ given by discarding the projection to $E_1'$. This gives us two morphisms $\llin\rightrightarrows\tanga{\derma\cap\derna}$ that are homotopically equivalent. Taking their difference and multiplying by coefficients from a partition of unity subordinate to the two charts we obtain one subcomplex on $\derma\cup\derna$ that represents the derived foliation. Since the entire derived scheme is second countable, continuing in this way we obtain 

\begin{proposition} For any purely derived foliation on $\quack{\demquot}{\smoup}$, that is Lagrangian with respect to $\imfo{\symstra}$ and negative definite with respect to $\refo{\symstra}$, whose existence is provided by Prop.\@ \ref{existence}, there is a representative given in terms of a globally defined subcomplex $\lmin\subseteq\tanga{\demquot}$, having generators in degrees $\geq 1$ and surjective in degrees $\geq 2$.\end{proposition}

Dividing by this $\lmin$ (in a naive way) we obtain a $\smoup$-linearized bundle on $\demquot[0]$ and a $\smoup$-invariant section. There is also a $\smoup$-invariant section of the dual bundle, whose derived critical locus gives us back the entire stack $\quack{\demquot}{\smoup}$. Details of this construction will be presented elsewhere.

\noindent{\small{\tt{dennis.borisov@uwindsor.ca, lkatzarkov@gmail.com, artan.sheshmani@gmail.com, yau@math.harvard.edu}}

\end{document}

